\documentclass[psamsfonts]{amsart}  
\usepackage{amssymb}

\usepackage[dvips]{graphicx}

\usepackage{graphicx}
\usepackage{amssymb}                       
\usepackage{amsmath}
\usepackage{color}
\usepackage{times}

\usepackage[unicode,bookmarks,colorlinks]{hyperref}
\hypersetup{
   linkcolor=brickred,
}

\definecolor{mahogany}{cmyk}{0, 0.77, 0.87, 0}
\definecolor{salmon}{cmyk}{0, 0.53, 0.38, 0}
\definecolor{melon}{cmyk}{0, 0.46, 0.50, 0}
\definecolor{yellowgreen}{cmyk}{0.44, 0, 0.74, 0}
\definecolor{brickred}{cmyk}{0, 0.89, 0.94, 0.28}
\definecolor{OliveGreen}{cmyk}{0.64, 0, 0.95, 0.40}
\definecolor{RawSienna}{cmyk}{0, 0.72, 1.0, 0.45}
\definecolor{ZurichRed}{rgb}{1, 0, 0} 

\usepackage{fancyhdr}
\usepackage{amsmath,amstext,amssymb,amsopn,amsthm}
\usepackage{amsmath,amssymb,amsthm}
\usepackage[mathscr]{eucal}

\newcommand{\CM}{\mathcal{M}}

\newcommand{\cadlag}{c\`adl\`ag}

\newcommand{\calP}{\mathcal{P}}
\newcommand{\trip}{(M, \CM, \mu)}
\newcommand{\CL}{\mathcal{L}}

\newcommand{\p}{{\mathbb{P}}}

\reversemarginpar

\definecolor{rb}{rgb}{0.1,0.2, 0.7}

\newtheorem{thm}{Theorem}[section]

\newtheorem{definition}{Definition}

\begin{document}

\newtheorem{lemma}[thm]{Lemma}
\newtheorem{proposition}{Proposition}
\newtheorem{theorem}{Theorem}[section]
\newtheorem{deff}[thm]{Definition}
\newtheorem{case}[thm]{Case}
\newtheorem{prop}[thm]{Proposition}
\newtheorem{example}{Example}

\newtheorem{corollary}{Corollary}

\theoremstyle{definition}
\newtheorem{remark}{Remark}

\numberwithin{equation}{section}
\numberwithin{definition}{section}
\numberwithin{corollary}{section}

\numberwithin{theorem}{section}

\numberwithin{remark}{section}
\numberwithin{example}{section}
\numberwithin{proposition}{section}

\newcommand{\gap}{\lambda_{2,D}^V-\lambda_{1,D}^V}
\newcommand{\gapR}{\lambda_{2,R}-\lambda_{1,R}}
\newcommand{\bD}{\mathrm{I\! D\!}}
\newcommand{\calD}{\mathcal{D}}
\newcommand{\calA}{\mathcal{A}}

\newcommand{\conjugate}[1]{\overline{#1}}
\newcommand{\abs}[1]{\left| #1 \right|}
\newcommand{\cl}[1]{\overline{#1}}
\newcommand{\expr}[1]{\left( #1 \right)}
\newcommand{\set}[1]{\left\{ #1 \right\}}

\newcommand{\calC}{\mathcal{C}}
\newcommand{\calE}{\mathcal{E}}
\newcommand{\calF}{\mathcal{F}}
\newcommand{\Rd}{\mathbb{R}^d}
\newcommand{\BR}{\mathcal{B}(\Rd)}
\newcommand{\R}{\mathbb{R}}
\newcommand{\T}{\mathbb{T}}
\newcommand{\D}{\mathbb{D}}

\newcommand{\al}{\alpha}
\newcommand{\RR}[1]{\mathbb{#1}}
\newcommand{\bR}{\mathrm{I\! R\!}}
\newcommand{\ga}{\gamma}
\newcommand{\om}{\omega}
\newcommand{\A}{\mathbb{A}}
\newcommand{\bH}{\mathbb{H}}

\newcommand{\bb}[1]{\mathbb{#1}}
\newcommand{\bI}{\bb{I}}
\newcommand{\bN}{\bb{N}}

\newcommand{\uS}{\mathbb{S}}
\newcommand{\M}{{\mathcal{M}}}
\newcommand{\calB}{{\mathcal{B}}}

\newcommand{\W}{{\mathcal{W}}}

\newcommand{\m}{{\mathcal{m}}}

\newcommand {\mac}[1] { \mathbb{#1} }

\newcommand{\bC}{\Bbb C}

\newtheorem{rem}[theorem]{Remark}
\newtheorem{dfn}[theorem]{Definition}
\theoremstyle{definition}
\newtheorem{ex}[theorem]{Example}
\numberwithin{equation}{section}

\newcommand{\Pro}{\mathbb{P}}
\newcommand\F{\mathcal{F}}
\newcommand\E{\mathbb{E}}
\newcommand\e{\varepsilon}
\def\H{\mathcal{H}}
\def\t{\tau}

\title[Weighted inequalities]{Sharp Weighted $L^2$ inequalities for square functions}

\author{Rodrigo Ba\~nuelos}\thanks{R. Ba\~nuelos is supported in part  by NSF Grant
\# 0603701-DMS}
\address{Department of Mathematics, Purdue University, West Lafayette, IN 47907, USA}
\email{banuelos@math.purdue.edu}
\author{Adam Os\c ekowski}\thanks{A. Os\c ekowski is supported in part by the NCN grant DEC-2014/14/E/ST1/00532.}
\address{Department of Mathematics, Informatics and Mechanics, University of Warsaw, Banacha 2, 02-097 Warsaw, Poland}
\email{ados@mimuw.edu.pl}


\begin{abstract}
Using Bellman function approach, we present new proofs of weighted $L^2$ inequalities for square functions, with the optimal dependence on the $A_2$ characteristics of the weight and further explicit constants. We study the estimates both in the analytic and probabilistic context, and, as application, obtain related estimates for the classical Lusin and Littlewood-Paley square functions.  
\end{abstract}

\maketitle

\section{Introduction}
Square function inequalities play an important role in both classical and noncommutative probability theory, harmonic analysis, potential theory and many other areas of mathematics.  The purpose of this paper is to establish sharp bounds in the dyadic case, which are closely related to the works of Bollob\'as \cite{Bo}, Davis \cite{D}, John and Nirenberg \cite{JN}, Littlewood \cite{L}, Marcinkiewicz \cite{Ma}, Paley \cite{P}, Slavin and Vasyunin  \cite{SV}, Wang \cite{W} and many others.

Let us start the paper with introducing some background and notation. In what follows, the interval $[0,1]$ will be denoted by $\mathfrak{I}$. Let $(h_n)_{n\geq 0}$ be the Haar system on $\mathfrak{I}$, that is, the family of functions given by
\begin{align*}
 &h_0=\chi_{[0,1]}, &&h_1=\chi_{[0,1/2)}-\chi_{[1/2,1)},\\
&h_2=\chi_{[0,1/4)}-\chi_{[1/4,1/2)},  &&h_3=\chi_{[1/2,3/4)}-\chi_{[3/4,1)},\\
& h_4=\chi_{[0,1/8)}-\chi_{[1/8,1/4)}, &&h_5=\chi_{[1/4,3/8)}-\chi_{[3/8,1/2)},
\end{align*}
and so on. For any dyadic subinterval $I$ of $\mathfrak{I}$ and any integrable function $\varphi:\mathfrak{I}\to \R$, we will write $\langle \varphi \rangle_I$ for the average of $\varphi$ over $I$: that is, $\langle \varphi\rangle_I=\frac{1}{|I|}\int_I \varphi$ (unless stated otherwise, the integration is with respect to Lebesgue measure). Furthermore, for any such $\varphi$ and any nonnegative integer $n$, we will write
$$ \varphi_n=\sum_{k=0}^{2^n-1} \frac{1}{|I_k|}\int_\mathfrak{I}\varphi(s)h_k(s)\mbox{d}s\,h_k$$
for the projection of $\varphi$ on the subspace generated by the first $2^n$ Haar functions ($I_k$ is the support of $h_k$). 
We define the dyadic square function of $\varphi$ by the formula
$$ S(\varphi)(x)=\left(\sum \left|\frac{1}{|I_n|}\int_\mathfrak{I} \varphi(s)h_n(s)\mbox{d}s\right|^2 \right)^{1/2},$$
where the summation runs over all nonnegative integers $n$ such that $x\in I_n$. 

The inequalities comparing the sizes of $\varphi$ and its square function $S(\varphi)$ are of importance in analysis and probability, and have been studied intensively in the literature. A classical result of Paley \cite{P} and Marcinkiewicz \cite{M} states that there are finite absolute constants $c_p$ ($0<p<\infty$) and $C_p$ ($1<p<\infty$), such that for any $\varphi:\mathfrak{I}\to\R$,
\begin{equation}\label{sq1}
  ||\varphi||_{L^p(\mathfrak{I})}\leq c_p||S(\varphi)||_{L^p(\mathfrak{I})}
\end{equation}
and
\begin{equation}\label{sq2}
 ||S(\varphi)||_{L^p(\mathfrak{I})}\leq C_p||\varphi||_{L^p(\mathfrak{I})}.
 \end{equation}
The question about the optimal values of $c_p$ and $C_p$ was studied by Davis \cite{D}. For $0<p<\infty$, let $\nu_p$ denote the smallest positive zero of a confluent hypergeometric function $M_p$ and let $\mu_p$ be the largest positive zero of the parabolic cylinder function of order $p$ (see Abramovitz and Stegun \cite{AS} for details). Using a related estimate for continuous-time martingales and Skorokhod embedding theorems, Davis \cite{D} showed that if $0< p\leq 2$, then the best choice for $c_p$ is $\nu_p$, while for $p\geq 2$, the optimal value of $C_p$ is $\nu_p^{-1}$. See also Wang \cite{W} for the vector-valued analogues of these results.

In recent years, a question about the weighted version of \eqref{sq1} and \eqref{sq2} gathered a lot of interest. In what follows, the  word ``weight'' refers to a locally integrable, positive function  on $\R$, which will usually be denoted by $w$. Given $p\in (1,\infty)$, we say that $w$ belongs to the Muckenhoupt $A_p$ class (or, in short, that $w$ is an $A_p$ weight), if the $A_p$ characteristics $[w]_{A_p}$, given by
$$ [w]_{A_p}:=\sup_I \left(\frac{1}{|I|}\int_I w \right)\left(\frac{1}{|I|}\int_I w^{-1/(p-1)}\right)^{p-1},$$
is finite. 
One can also define the appropriate versions of this condition for $p=1$ and $p=\infty$, by passing above with $p$ to the appropriate limit (see e.g. \cite{Gr}, \cite{Hr}). However, we omit the details, as in this paper we will be mainly concerned with the case $1<p<\infty$. The condition $A_p$ arises naturally in the study of weighted estimates for the Hardy-Littlewood maximal operator, as Muckenhoupt showed in \cite{M}. 

Coming back to square function  estimates, the first weighted bound in this setting is due to Buckley \cite{Bu}, who showed the $L^2$ inequality
\begin{equation}\label{buckley}
 ||S(\varphi)||_{L^2_w(\mathfrak{I})}\leq C[w]_{A_2}^{3/2}||\varphi||_{L^2_w(\mathfrak{I})},
\end{equation}
with $C$ being a universal constant. 
Here, of course, the weighted $L^2$ norm is given by
$$||\varphi||_{L^2_w(\mathfrak{I})}=\left(\int_\mathfrak{I} \varphi^2 w\right)^{1/2}.$$
Can the exponent $3/2$ in \eqref{buckley} be decreased? This question was studied by Hukovic \cite{H} and Hukovic, Treil and Volberg \cite{HTV}. It turns out that the sharp dependence is linear, i.e., the best exponent is $1$. This result was later reproved by Wittwer \cite{Wi} and Petermichl and Pott in \cite{PP} using a different approach. Actually, the latter paper contains also the proof of the reverse inequality
$$ ||\varphi||_{L^2_w(\mathfrak{I})}\leq C[w]_{A_2}^{1/2}||S(\varphi)||_{L^2_w(\mathfrak{I})},$$
in which the exponent $1/2$ is also optimal. A considerable extension of these results was obtained recently by Cruz-Uribe, Martell and P\'erez in \cite{CMP}, who showed the weighted $L^p$ bound 
$$ ||S(\varphi)||_{L^p_w(\mathfrak{I})}\leq C[w]_{A_p}^{\max\{(p-1)^{-1},1/2\}}||\varphi||_{L^p_w(\mathfrak{I})}$$
and proved that the exponent $\max\{(p-1)^{-1},1/2\}$ is the best.

One of our objectives is to give a yet another proof of the weighted $L^2$ estimate for square functions. Our reasoning will rest on the construction of certain special functions which enjoy appropriate majorization and concavity properties.  This type of approach, called the Bellman function technique, originates from the theory of optimal stochastic control and has turned out to be very efficient in various problems in analysis and probability.

One of our main results is the following.

\begin{thm}\label{weights}
Suppose that $w$ is an $A_2$ weight and $\varphi$ is a function belonging to $L^2_w(\mathfrak{I})$. Then we have the estimates
\begin{equation}\label{mainin}
||\varphi||_{L^2_w(\mathfrak{I})}\leq (160[w]_{A_2})^{1/2}||S(\varphi)||_{L^2_w(\mathfrak{I})}
\end{equation}
and
\begin{equation}\label{mainin6}
||S(\varphi)||_{L^2_w(\mathfrak{I})}\leq 8\sqrt{2}[w]_{A_2}||\varphi||_{L^2_w(\mathfrak{I})}.
\end{equation}
Furthermore, 
\begin{equation}\label{mainin2}
 ||S(\varphi)||_{L^2_w(\mathfrak{I})}\leq \inf_{1<r<2}\left(\frac{2r}{2-r}[w]_{A_r}\right)^{1/2}||\varphi||_{L^2_w(\mathfrak{I})}.
\end{equation}
\end{thm}
The reason why we have included \eqref{mainin2} in the above statement is that this estimate implies the weighed $L^2$ bound with the linear dependence on $[w]_{A_2}$, and hence can be regarded as an improvement of \eqref{mainin6}. To see the implication, recall the following classical fact, due to Coifman and Fefferman \cite{CF}. 

\begin{lemma}\label{CFlemma}
There is a constant $\kappa_{p}$ depending only on $p$ such that the following holds. If $w$ is an $A_p$ weight ($1<p<\infty$) on an interval, then $w$ is an $A_{p-\e}$ weight, where $\e= \kappa_{p}^{-1}[w]_{A_p}^{-1/(p-1)}$. Moreover, we have $[w]_{A_{p-\e}}\leq \kappa_{p}[w]_{A_p}$.
\end{lemma}

We apply this lemma with $p=2$.   Taking  $r=2-\e=2-\kappa_2^{-1}[w]_{A_2}^{-1}$, the estimate \eqref{mainin2} yields
 $$||S(\varphi)||_{L^2_w(\mathfrak{I})}\leq \left(\frac{2r}{2-r}[w]_{A_r}\right)^{1/2}||\varphi||_{L^2_w(\mathfrak{I})}\leq 2\kappa_2[w]_{A_2}||\varphi||_{L^2_w(\mathfrak{I})},$$
 as desired.
 
We have organized the paper as follows. In the next section we introduce the Bellman functions corresponding to the estimates \eqref{mainin}, \eqref{mainin6} and \eqref{mainin2}, and study their properties. \S\ref{proofthem1} is devoted to the proof of Theorem \ref{weights}. In \S\ref{mart-wights}, we establish an appropriate probabilistic analogue of Theorem \ref{weights} for continuous time martingales  and then, in \S\ref{L-P}  show how this  yields similar results  for the classical Lusin and Littlewood-Paley square functions.  In section \S\ref{Markovian} we further elaborate on extensions to more general Markovian semigroups.  

\section{Special functions}

Throughout this section, $c>1$ is a fixed parameter. For any $1<r<\infty$, the symbol $\Omega_c^r$ will denote the hyperbolic domain given by
$$ \Omega_c^r=\{(w,v)\in \bR_{+} \times \bR_{+}: 1\leq wv^{r-1}\leq c\}.$$

\subsection{Bellman function corresponding to \eqref{mainin}}
The key role in the proof of the estimate \eqref{mainin} is played by the function $B_c:\R\times [0,\infty)\times \Omega_c^2\to \R$, given by
$$ B_c(x,y,w,v)=x^2w\varphi(wv)-40cyw,$$
where
$$ \varphi(t)=2-\frac{1}{t}-\frac{\ln t}{2c},\qquad t\in [1,c].$$ 
In what follows, usually we will skip the lower index and write $B$ instead of $B_c$, but keep in mind that the function does depend on the parameter $c$. 
Let us study some simple majorization properties of $B$. We start with the trivial observation that $\varphi(t)\leq 2$ for all $t$, which implies the bound
\begin{equation}\label{init}
 B(x,x^2,w,v)\leq 2x^2w-40cx^2w\leq 0 \qquad \mbox{for all }x\in \R,\,(w,v)\in \Omega_c^2.
\end{equation}
 Next, note that for any $t\in [1,c]$ we have $\varphi(t)\geq 2-1-\frac{1}{2}=\frac{1}{2}$, and hence
\begin{equation}\label{maj}
B(x,y,w,v)\geq \frac{1}{2}w(x^2-80cy)\qquad \mbox{for all }(x,y,w,v)\in \R\times [0,\infty)\times \Omega_c^2.
\end{equation}
We turn our attention to the crucial property of $B$. It can be regarded as a concavity-type condition.

\begin{lemma}\label{convlemma}
Suppose that $(x,y,w,v)\in \R\times[0,\infty)\times\Omega_c^2$ is a given point and assume further that $e,\,f$ are real numbers such that the line segment with endpoints $(w\pm e,v\pm f)$ is entirely contained in $\Omega_c^2$. Then for any $d\in \R$ we have
\begin{equation}\label{conv}
\begin{split}
 &2B(x,y,z,w)\\
&\geq B(x-d,y+d^2,w-e,v-f)+B(x+d,y+d^2,w+e,v+f).
\end{split}
\end{equation}
\end{lemma}
\begin{proof}
Introduce the function $b:\R\times \Omega_c^2\to [0,\infty)$, given by $ b(x,w,v)=x^2w\varphi(wv).$ Of course, we have the identity
$$ B(x,y,w,v)=b(x,w,v)-40cyw.$$
Since $ -c(y+d^2)(w-e)-c(y+d^2)(w+e)=-2cyw-2cd^2w,$ we see that the claim is equivalent to
\begin{equation}\label{weath}
 2b(x,w,v)\geq b(x-d,w-e,v-f)+b(x+d,w+e,v+f)-80cd^2w.
\end{equation}
To show this, we will prove that the matrix
\begin{equation}\label{D^2}
\mathbb{A}(x,y,w)= D^2b(x,w,v)-\left[\begin{array}{ccc}
80cw & 0 & 0\\
0 & 0 & 0\\
0 & 0 & 0
\end{array}\right]
\end{equation}
is nonpositive-definite. To see how this yields \eqref{weath}, consider the function
$$ F(t)= b(x+td,w+te,v+tf)-40ct^2d^2w,\qquad t\in [-1,1].$$
Note that for each such $t$, the point $(x+td,w+te,v+tf)$ lies in the domain of $b$: see the assumption in the statement of the lemma above.
Denoting the vector $(d,e,f)$ by $\Delta$, we compute that
\begin{align*}
 F''(t)+F''(-t)&=\left(D^2b(x+td,w+te,v+tf)\Delta,\Delta\right)-80cwd^2\\
&\quad +\left(D^2b(x-td,w-te,v-tf)\Delta,\Delta\right)-80cwd^2\\
&=\left(D^2b(x+td,w+te,v+tf)\Delta,\Delta\right)-80c(w+te)d^2\\
&\quad +\left(D^2b(x-td,w-te,v-tf)\Delta,\Delta\right)-80c(w-te)d^2\\
&=(\mathbb{A}(x+td,w+te,v+tf)\Delta,\Delta)\\
&\quad +(\mathbb{A}(x-td,w-te,v-tf)\Delta,\Delta)\\
&\leq 0.
\end{align*}
Consequently, $F(1)+F(-1)\leq F(0)$, which is precisely \eqref{weath}.  To show that $\mathbb{A}(x,w,v)$ is nonpositive-definite, we compute directly that the matrix is equal to
$$\left[\begin{array}{ccc}
2w\varphi(t)-80cw & 2x\varphi+2xt\varphi'(t) & 2xw^2\varphi'(t)\\
2x\varphi(t)+2xt\varphi'(t) & 2x^2v\varphi'(t)+2x^2tv\varphi''(t) & 2x^2w\varphi'(t)+x^2tw\varphi''(t)\\
2xw^2\varphi'(t) & 2x^2w\varphi'(t)+x^2tw\varphi''(t) & x^2w^3\varphi''(t)
\end{array}\right],$$
where $t=wv$. By well-known facts from linear algebra, it is enough to show that
\begin{equation}\label{1}
 x^2w^3\varphi''(t)\leq 0,
 \end{equation}
\begin{equation}\label{2}
\operatorname*{det}\left[
\begin{array}{cc}
2x^2v\varphi'(t)+2x^2tv\varphi''(t) & 2x^2w\varphi'(t)+x^2tw\varphi''(t)\\
2x^2w\varphi'(t)+x^2tw\varphi''(t) & x^2w^3\varphi''(t)
\end{array}
\right]\leq 0
 \end{equation}
and
\begin{equation}\label{3}
\operatorname*{det}\mathbb{A}(x,w,v)\leq 0.
 \end{equation}
To establish \eqref{1}, observe that $t=wv\in [1,c]$, by the definition of $\Omega_c^2$, and therefore
$$ x^2w^3\varphi''(t)=-\frac{2x^2w^3}{2ct^3}(4c-t)\leq 0.$$
The inequality \eqref{2} is equivalent to $ 2 \varphi'(t)(2\varphi'(t)+t\varphi''(t))\leq 0$, and follows from the estimates
$$ \varphi'(t)=\frac{1}{2ct^2}\left(2c-t\right)\geq 0,\qquad 2\varphi'(t)+t\varphi''(t)=-\frac{1}{2ct}\leq 0.$$
Finally, we turn our attention to \eqref{3}. Let us simplify the matrix $\mathbb{A}$, by carrying out some elementary operations. Dividing the second row and the second column by $x$, and then the third row and column by $xw$, we see that the determinant of $\mathbb{A}$ has the same sign as
$$\operatorname*{det}\left[\begin{array}{ccc}
2w\varphi(t)-80cw & 2\varphi+2t\varphi'(t) & 2w\varphi'(t)\\
2\varphi(t)+2t\varphi'(t) & 2v\varphi'(t)+2tv\varphi''(t) & 2\varphi'(t)+t\varphi''(t)\\
2w\varphi'(t) & 2\varphi'(t)+t\varphi''(t) & w\varphi''(t)
\end{array}\right].$$
Next, multiply the third row by $v$ and subtract it from the second row; then multiply the second row by $w$ and subtract it from the first row. As the result, we obtain that the sign of the determinant of $\mathbb{A}$ is the same as that of
\begin{align*}
&\operatorname*{det}\left[\begin{array}{ccc}
-80cw & 2\varphi+2t\varphi'(t) & 0\\
2\varphi(t) & 0 & 2\varphi'(t)\\
2w\varphi'(t) & 2\varphi'(t)+t\varphi''(t) & w\varphi''(t)
\end{array}\right]\\
&=4w\left[\left(2\left(\varphi'(t)\right)^2-\varphi(t)\varphi''(t)\right)(\varphi(t)+t\varphi'(t))+40c\varphi'(t)\left(2\varphi'(t)+t\varphi''(t)\right)\right].
\end{align*}
However, we compute that
$$ \varphi(t)+t\varphi'(t)=2-\frac{t}{2c}-\frac{1}{2c}\leq 2,$$
$$ 2(\varphi'(t))^2=\varphi'(t)\cdot \frac{2c-t}{ct^2}\leq \frac{2\varphi'(t)}{t}$$
and, since $\varphi(t)\leq 2$, 
$$ -\frac{\varphi(t)\varphi''(t)}{\varphi'(t)}\leq \frac{2\left(\frac{2}{t^3}-\frac{1}{2ct^2}\right)}{\frac{1}{t^2}-\frac{1}{2ct}}\leq \frac{8}{t}.$$
Consequently,
$$ \left(2\left(\varphi'(t)\right)^2-\varphi(t)\varphi''(t)\right)(\varphi(t)+t\varphi'(t))\leq \frac{20\varphi'(t)}{t}$$
and since
$$40c\varphi'(t)\left(2\varphi'(t)+t\varphi''(t)\right)=-\frac{20\varphi'(t)}{t},$$
the inequality \eqref{3} is satisfied. This completes the proof.
\end{proof}

\subsection{Bellman function corresponding to \eqref{mainin2}}

The Bellman function associated with the $A_r$-estimate is slightly simpler. Let $r$ be an arbitrary number belonging to $(1,2)$ and define $B=B_{c,r}:\R\times [0,\infty)\times \Omega_c^r\to \R$ by
$$ B(x,y,w,v)=yw-\frac{rc}{2-r}\frac{x^2}{v^{r-1}}.$$
As previously, we will first establish the appropriate majorizations for $B$. By the definition of $\Omega_c^r$, we have $cv^{1-r}\geq w$ and hence
\begin{equation}\label{init2}
B(x,x^2,w,v)\leq x^2w\left(1-\frac{r}{2-r}\right)\leq 0 \qquad \mbox{for all }x\in \R,\,(w,v)\in \Omega_c^r,
\end{equation}
where in the last bound we used the estimate $r>1$. Furthermore, the inequality $v^{1-r}\leq w$ implies
\begin{equation}\label{maj2}
B(x,y,w,v)\geq yw-\frac{rc}{2-r}x^2w\qquad \mbox{for all }(x,y,w,v)\in \R\times [0,\infty)\times \Omega_c^r. 
\end{equation}
We turn to the analogue of Lemma \ref{convlemma}.

\begin{lemma}\label{convlemma2}
Suppose that $(x,y,w,v)\in \R\times[0,\infty)\times\Omega_c^r$ is a given point and assume further that $e,\,f$ are real numbers such that the line segment with endpoints $(w\pm e,v\pm f)$ is entirely contained in $\Omega_c^r$. Then for any $d\in \R$ we have
\begin{equation}\label{conv2}
\begin{split}
 &2B(x,y,z,w)\\
&\geq B(x-d,y+d^2,w-e,v-f)+B(x+d,y+d^2,w+e,v+f).
\end{split}
\end{equation}
\end{lemma}
\begin{proof}
Repeating the reasoning from the proof of Lemma \ref{convlemma}, we see that it is enough to show that the matrix
$$ \mathbb{A}(x,w,v)=D^2b(x,w,v)+\left[\begin{array}{ccc}
2w & 0 & 0\\
0 & 0 & 0\\
0 & 0 & 0
\end{array}\right]$$
is nonpositive-definite. Here $b$ is a function given on $\R\times \Omega_c^r$ by the formula
$$ b(x,w,v)=-\frac{rc}{2-r}\frac{x^2}{v^{r-1}}.$$
We compute directly that
$$ \mathbb{A}(x,w,v)=\left[\begin{array}{ccc}
2w-\frac{2rc}{(2-r)v^{r-1}} & 0 & \frac{2r(r-1)cx}{(2-r)v^{r}}\\
0 & 0 & 0\\
\frac{2r(r-1)cx}{(2-r)v^{r}} & 0 & -\frac{r^2(r-1)c}{2-r}\frac{x^2}{v^{r+1}}
\end{array}\right].$$
We see that the entry in the lower-right corner is nonpositive and the determinant is equal to $0$. Thus it is enough to show that
$$ \operatorname*{det}\left[\begin{array}{cc}
2w-\frac{2rc}{(2-r)v^{r-1}} &  \frac{2r(r-1)cx}{(2-r)v^{r-2}}\\
\frac{2r(r-1)cx}{(2-r)v^{r-2}} &  -\frac{r^2(r-1)c}{2-r}\frac{x^2}{v^{r+1}}
\end{array}\right]\geq 0.$$
But $w\leq cv^{1-r}$, so the above determinant is not smaller than
\begin{align*}
& \operatorname*{det}\left[\begin{array}{cc}
\frac{2c}{v^{r-1}}-\frac{2rc}{(2-r)v^{r-1}} &  \frac{2r(r-1)cx}{(2-r)v^{r-2}}\\
\frac{2r(r-1)cx}{(2-r)v^{r-2}} &  -\frac{r^2(r-1)c}{2-r}\frac{x^2}{v^{r+1}}
\end{array}\right]\\
&=\frac{2c^2x^2}{v^{2r}}\left[-\left(1-\frac{r}{2-r}\right)\frac{r^2(r-1)}{2-r}-\frac{2r^2(r-1)^2}{(2-r)^2}\right]=0.
\end{align*}
This completes the proof.
\end{proof}

\subsection{The special function corresponding to \eqref{mainin6}}\label{alternative} Finally, we turn our attention to the second weighted estimate of Theorem \ref{weights}. 
The Bellman function is slightly more complicated than that studied in the preceding section, but it has the advantage that it produces a ,,self-contained'' proof of \eqref{mainin6} (i.e., it does not refer to the self-improving properties of $A_2$ weights). Define $B=B_c:\mathbb{R}\times [0,\infty)\times \Omega^2_c\to \R$ given by the formula
$$ B(x,y,w,v)=yw-\frac{16c^2 x^2w}{(wv-1/2)^\alpha},$$
where $\alpha=1-(4c)^{-1}$. Let us now establish the appropriate majorizations for this object. First, note that
$$ B(x,x^2,w,v)\leq x^2w\left[1-\frac{16c^2}{(c-1/2)^\alpha}\right]\leq x^2w(1-16c)\leq 0.$$
Furthermore, for any $(x,y,w,v)$ from the domain of $B$, we clearly have
$$ B(x,y,w,v)\geq yw- 2^\alpha\cdot 16 c^2 x^2w\geq yw-32 c^2x^2w.$$
Finally, we will show that $B$ enjoys the property described in Lemma \ref{convlemma}. We have
$$ B(x,y,w,v)=16c^2\left[\frac{yw}{16c^2}+b(x,w,v)\right],$$
where $b(x,w,v)=-x^2w(wv-1/2)^{-\alpha}$. Arguing as in the preceding subsections, we see that it is enough to prove that the matrix
$$ \mathbb{A}(x,w,v)=D^2b(x,w,v)+\left[\begin{array}{ccc} 2w/(16c^2) & 0 & 0\\
0 & 0 & 0\\
0 & 0 & 0
\end{array}\right]$$
is nonpositive-definite. Substituting $t=wv$, we compute that
$$ \mathbb{A}(x,w,v)=\left[
\begin{array}{ccc}
\frac{w}{8c^2}-\frac{2w}{(t-1/2)^\alpha} & -\frac{2x((1-\alpha)t-1/2)}{(t-1/2)^{\alpha+1}} & \frac{2\alpha x w^2}{(t-1/2)^{\alpha+1}}\\
-\frac{2x((1-\alpha)t-1/2)}{(t-1/2)^{\alpha+1}} & \frac{\alpha x^2v((1-\alpha)t-1)}{(t-1/2)^{\alpha+2}} & \frac{\alpha x^2w((1-\alpha)t-1)}{(t-1/2)^{\alpha+2}}\\
\frac{2\alpha x w^2}{(t-1/2)^{\alpha+1}} & \frac{\alpha x^2 w((1-\alpha)t-1)}{(t-1/2)^{\alpha+2}} & -\frac{\alpha(\alpha+1)x^2w^3}{(t-1/2)^{\alpha+2}}.
\end{array}\right].$$
As previously, we apply Sylvester's criterion. Obviously, we have 
$$ -\frac{\alpha(\alpha+1)x^2w^3}{(t-1/2)^{\alpha+2}}\leq 0.$$
Furthermore,
\begin{align*}
&\operatorname*{det}\left[\begin{array}{cc}
\frac{\alpha x^2v((1-\alpha)t-1)}{(t-1/2)^{\alpha+2}} & \frac{\alpha x^2w((1-\alpha)t-1)}{(t-1/2)^{\alpha+2}}\\
\frac{\alpha x^2 w((1-\alpha)t-1)}{(t-1/2)^{\alpha+2}} & -\frac{\alpha(\alpha+1)x^2w^3}{(t-1/2)^{\alpha+2}}
\end{array}\right]=\frac{2\alpha^2x^4w^2(1-(1-\alpha)t)(t-1/2)}{(t-1/2)^{2\alpha+4}}\geq 0,
\end{align*}
since $1-(1-\alpha)t=1-t/(4c)>0$. Thus, to show that $\mathbb{A}(x,w,v)$ is nonpositive-definite, it suffices to show that its determinant is nonpositive. To do this, let us conduct some operations on the rows of this matrix. First, multiply the third row by $v/w$ and add it to the second row. Then
$$ \operatorname{det}\mathbb{A}(x,w,v)=\operatorname*{det}
\left[
\begin{array}{ccc}
\frac{w}{8c^2}-\frac{2w}{(t-1/2)^\alpha} & -\frac{2x((1-\alpha)t-1/2)}{(t-1/2)^{\alpha+1}} & \frac{2\alpha x w^2}{(t-1/2)^{\alpha+1}}\\
-\frac{2x}{(t-1/2)^\alpha} & 0 & \frac{2\alpha x^2w}{(t-1/2)^{\alpha+1}}\\
\frac{2\alpha x w^2}{(t-1/2)^{\alpha+1}} & \frac{\alpha x^2 w((1-\alpha)t-1)}{(t-1/2)^{\alpha+2}} & -\frac{\alpha(\alpha+1)x^2w^3}{(t-1/2)^{\alpha+2}}.
\end{array}\right].$$
Next, multiply the second row by $\alpha w^2/(t-1/2)$ and add it to the third row; furthermore, multiply the second row by $ w/x$ and subtract it from the first row. As the result, we see that
$$ \operatorname{det}\mathbb{A}(x,w,v)=\operatorname*{det}
\left[
\begin{array}{ccc}
\frac{w}{8c^2} & -\frac{2x((1-\alpha)t-1/2)}{(t-1/2)^\alpha} & 0\\
-\frac{2x}{(t-1/2)^\alpha} & 0 & \frac{2\alpha x^2w}{(t-1/2)^{\alpha+1}}\\
0 & \frac{\alpha x^2 w((1-\alpha)t-1)}{(t-1/2)^{\alpha+2}} & \frac{\alpha(\alpha-1)x^2w^3}{(t-1/2)^{\alpha+2}}.
\end{array}\right],$$
so the sign of $\operatorname*{det}\mathbb{A}(x,w,v)$ is the same as the sign of
\begin{align*}
& \operatorname*{det}
\left[
\begin{array}{ccc}
\frac{w}{16c^2} & -((1-\alpha)t-1/2) & 0\\
-1 & 0 & 1\\
0 & \alpha((1-\alpha)t-1) & \frac{(\alpha-1)w}{(t-1/2)^{\alpha}}
\end{array}\right]\\
&\qquad \qquad =\frac{((1-\alpha)t-1/2)(1-\alpha)}{(t-1/2)^\alpha}+\frac{\alpha(1-(1-\alpha)t)}{16c^2}.
\end{align*}
But the above expression is nonpositive; this can be equivalently rewritten in the form
$$ \alpha(1-(1-\alpha)t)(t-1/2)^\alpha<(1-\alpha)(1/2-(1-\alpha)t)\cdot 16c^2$$
and follows from the observations that $\alpha<1$, $0<1-(1-\alpha)t<1$, $(t-1/2)^\alpha<c^\alpha\leq c$ and $  (1-\alpha)(1/2-(1-\alpha)t)\cdot 16c^2=4c(1/2-t/(4c))\geq 4c(1/2-1/4)=c.$

\section{Proof of Theorem \ref{weights}}\label{proofthem1}
We start with the following geometric fact.

\begin{lemma}\label{geomlemma}
Assume that $c>1$ and $r\in (1,2]$. 
Suppose that points $\mathtt{P}$, $\mathtt{Q}$ and $\mathtt{R}=\alpha \mathtt{P}+(1-\alpha)\mathtt{Q}$ lie in $\Omega_c^r$. Then the whole line segment $\mathtt{P}\mathtt{Q}$ is contained within $\Omega_{2c}^r$.
\end{lemma}
\begin{proof}
Using a simple geometrical argument, it is enough to consider the case when the points  $\mathtt{P}$ and $\mathtt{R}$ lie on the curve $wv^{r-1}=c$ (the upper boundary of $\Omega_c^r$) and $\mathtt{Q}$ lies on the curve $wv^{r-1}=1$ (the lower boundary of $\Omega_c^r$). Then the line segment $\mathtt{R}\mathtt{Q}$ is contained within $\Omega_c^r$, and hence also within $\Omega_{2c}^r$, so  it is enough to ensure that the segment $\mathtt{P}\mathtt{R}$ is contained in $\Omega_{2c}^r$. Let $\mathtt{P}=(\mathtt{P}_x,\mathtt{P}_y)$, $\mathtt{Q}=(\mathtt{Q}_x,\mathtt{Q}_y)$ and $\mathtt{R}=(\mathtt{R}_x,\mathtt{R}_y)$. We consider two cases. If $\mathtt{P}_x<\mathtt{R}_x$, then 
$$ \mathtt{P}_y=2\mathtt{R}_y-\mathtt{Q}_y<2\mathtt{R}_y,$$
so the segment $\mathtt{P}\mathtt{R}$ is contained in the quadrant $\{(x,y):x\leq \mathtt{R}_x,\,y\leq 2\mathtt{R}_y\}$. Consequently, $\mathtt{P}\mathtt{R}$ lies below the hyperbola $xy^{r-1}=2^{r-1}c$ passing through $(\mathtt{R}_x,2\mathtt{R}_y)$, and hence also below the hyperbola $xy^{r-1}=2c$. This proves the assertion in the case $\mathtt{P}_x<\mathtt{R}_x$. In the case $\mathtt{P}_x\geq\mathtt{R}_x$ the reasoning is similar. Indeed, we check easily that the line segment $\mathtt{P}\mathtt{R}$ lies below the hyperbola $xy^{r-1}=2c$ passing through $(2\mathtt{R}_x,\mathtt{R}_y)$.
\end{proof}

We are ready to establish the inequalities of Theorem \ref{weights}. 

\begin{proof}[Proof of \eqref{mainin}] Let us start with introducing some auxiliary objects and notation. 
Let $(\mathcal{I}^n)_{n\geq 0}$ denote the dyadic filtration of $\mathfrak{I}$: given a nonnegative integer $n$, $\mathcal{I}^n$ denotes the $\sigma$-algebra generated by all dyadic intervals contained within $\mathfrak{I}$, which are of measure $2^{-n}$. Let $w$ be an $A_2$ weight with $c=[w]_{A_2}$ and let $\varphi$ be a function belonging to $L^2_w(\mathfrak{I})$. For any $x\in \mathfrak{I}$ and any nonnegative integer $n$, define
\begin{equation}\label{deffn}
 \varphi_n(x)=\langle \varphi\rangle_I,\quad \mathtt{w}_n(x)=\langle  w\rangle_I\quad \mbox{ and }\quad \mathtt{v}_n(x)=\langle w^{-1}\rangle_I,
\end{equation}
where $I=I(x)$ is the atom of $\mathcal{I}^n$ which contains $x$ (such an atom is unique for almost all $x\in\mathfrak{I}$, so the above equalities give functions which are well-defined on the subset of $\mathfrak{I}$ of full measure). Note that this definition of $\varphi_n$ is consistent with that given in the introductory section. Furthermore, define the truncated square function of $\varphi$ by
$$ S_n(\varphi)=S(\varphi_n),\qquad n=0,\,1,\,2,\,\ldots.$$

Now, let $B=B_{2c}$ be the function introduced in \S2.1 (note that we take the index $2c$). A crucial fact which exhibits the interplay between $B$ and the sequences $\varphi_n$, $S_n(\varphi)$, $\mathtt{w}_n$ and $\mathtt{v}_n$, is that
\begin{equation}\label{monot}
\begin{split}
&\int_\mathfrak{I} B(\varphi_n(s),S_n^2(\varphi)(s),\mathtt{w}_n(s),\mathtt{v}_n(s))\mbox{d}s\\
&\qquad \qquad \qquad \geq \int_\mathfrak{I} B(\varphi_{n+1}(s),S^2_{n+1}(\varphi)(s),\mathtt{w}_{n+1}(s),\mathtt{v}_{n+1}(s))\mbox{d}s
\end{split}
\end{equation}
for $n\geq 0$. To prove this estimate, fix an $n$ and pick an atom $I$ of $\mathcal{I}^n$. We will prove a slightly stronger statement than \eqref{monot}, in which $\mathfrak{I}$ is replaced by $I$, i.e.,
\begin{equation}\label{monot2}
\begin{split}
&\int_I B(\varphi_n(s),S_n^2(\varphi)(s),\mathtt{w}_n(s),\mathtt{v}_n(s))\mbox{d}s\\
&\qquad \qquad \qquad \geq \int_I B(\varphi_{n+1}(s),S^2_{n+1}(\varphi)(s),\mathtt{w}_{n+1}(s),\mathtt{v}_{n+1}(s))\mbox{d}s.
\end{split}
\end{equation}
Denote the left and right half of $I$ by $I^-$ and $I^+$, respectively. Then the functions $\varphi_n$, $S_n^2(\varphi)$, $\mathtt{w}_n$ and $\mathtt{v}_n$ are constant on $I$; let us denote the corresponding values by $x$, $y$, $w$ and $v$. Similarly, $\varphi_{n+1}$, $S_{n+1}^2(\varphi)$, $\mathtt{w}_{n+1}$ and $\mathtt{v}_{n+1}$ are constant on each of $I^-$, $I^+$: denote the appropriate values by $x_\pm$, $y_\pm$, $w_\pm$ and $v_\pm$. From the very definition of the sequences $(\varphi_n)_{n\geq 0}$,  $(\mathtt{w}_n)_{n\geq 0}$ and $(\mathtt{v}_n)_{n\geq 0}$, we infer that
$$ x=(x_-+x_+)/2,\quad w=(w_-+w_-)/2\quad \mbox{and}\quad v=(v_-+v_+)/2,$$
so there are $d,\,e,\,f\in \R$ such that
$$ x_\pm=x\pm d,\quad w_\pm=w\pm e\quad \mbox{and}\quad v_\pm=v\pm f.$$
In addition, by the very definition of $(S_n(\varphi))_{n\geq 0}$, we see that $ y_-=y_+=y+d^2.$ If we plug all these facts into \eqref{monot2} and divide both sides by $|I|/2$, we get an estimate which is equivalent to \eqref{conv}. Thus, \eqref{monot2} will be established if we show that the assumption of Lemma \ref{convlemma} is satisfied. However, the points
\begin{align*}
 (w_-,v_-)&=\left(\frac{1}{|I^-|}\int_{I^-}w,\frac{1}{|I^-|}\int_{I^-}w^{-1}\right),\\
(w_+,v_+)&=\left(\frac{1}{|I^+|}\int_{I^+}w,\frac{1}{|I^+|}\int_{I^+}w^{-1}\right),\\
(w,v)&=\frac{(w_-,v_-)+(w_+,v_+)}2=\left(\frac{1}{|I|}\int_{I}w,\frac{1}{|I|}\int_{I}w^{-1}\right)
\end{align*}
belong to $\Omega_c^2$, by the very definition of $A_2$ weights. Consequently, by Lemma \ref{geomlemma}, the line segment with endpoints $(w_-,v_-)$, $(w_+,v_+)$ is entirely contained in $\Omega^2_{2c}$, which is precisely the requirement of Lemma \ref{convlemma} (recall that the special function $B$ we use corresponds to the parameter $2c$). This yields \eqref{monot2}, and summing over all atoms $I$ of $\mathcal{I}^n$, we obtain \eqref{monot}. The remainder of the proof is straightforward. By induction, \eqref{monot} gives
\begin{equation}\label{induct}
 \int_\mathfrak{I} B(\varphi_n(s),S_n^2(\varphi)(s),\mathtt{w}_n(s),\mathtt{v}_n(s))\mbox{d}s \leq \int_\mathfrak{I} B(\varphi_0(s),S_0^2(\varphi)(s),\mathtt{w}_0(s),\mathtt{v}_0(s))\mbox{d}s.
\end{equation}
However, we have
$$ \varphi_0(s)=\frac{1}{|\mathfrak{I}|}\int_\mathfrak{I} \varphi,\quad\mbox{and}\quad  S_0(\varphi)=\left|\frac{1}{|\mathfrak{I}|}\int_\mathfrak{I} \varphi\right|,$$
so \eqref{init} implies that the right-hand side of \eqref{induct} is nonpositive. To deal with the left-hand side, we exploit \eqref{maj}. As the result, we obtain the estimate
$$  \int_\mathfrak{I} (\varphi_n(s))^2\mathtt{w}_n(s)\mbox{d}s\leq 160 c\int_\mathfrak{I} S_n^2(\varphi)(s)\mathtt{w}_n(s)\mbox{d}s,$$
which, by the very definition of $\mathtt{w}_n$, implies
$$  \int_\mathfrak{I} (\varphi_n(s))^2w(s)\mbox{d}s\leq 160 c\int_\mathfrak{I} S_n^2(\varphi)(s)w(s)\mbox{d}s.$$
However, if we let $n$ go to infinity, then $S_n(\varphi)\uparrow S(\varphi)$ and $\varphi_n\to \varphi$ almost surely, by Lebesgue's differentiation theorem. Consequently, Fatou's lemma and Lebesgue's monotone convergence theorem imply
$$ ||\varphi||_{L^2_w(\mathfrak{I})}\leq (160c)^{1/2}||S(\varphi)||_{L^2_w(\mathfrak{I})},$$
which is the desired claim.
\end{proof}

\begin{proof}[Proof of \eqref{mainin6}] The arguments go along the same lines as above. We omit the straightforward repetitions.
\end{proof}

\begin{proof}[Proof of \eqref{mainin2}]
Here the reasoning is essentially the same as above, but we have decided to include some details due to the appearance of the $A_r$ weights. Suppose that $w$ is an $A_2$ weight. Then, by Lemma \ref{CFlemma}, $w$ is an $A_r$ weight for some $r<2$. Let $c=[w]_{A_r}$ and assume that $B=B_{2c,r}$ is the function introduced in \S2.2 (again, note that we use the doubled index $2c$). Let $\varphi$ be a function belonging to $L^2_w(\mathfrak{I})$ and define $(\varphi_n)_{n\geq 0}$, $(S_n(\varphi))_{n\geq 0}$, $(\mathtt{w}_n)_{n\geq 0}$ as previously. The corresponding sequence $(\mathtt{v}_n)_{n\geq 0}$ is slightly different, as it captures the fact that $w$ is an $A_r$ weight: let
$$ \mathtt{v}_n(x)=\langle w^{-1/(r-1)}\rangle_{I(x)}$$
(recall that $I(x)$ is the element of $\mathcal{I}^n$ which contains the point $x$).
By Muckenhoupt's condition $A_r$, we see that the sequence $((\mathtt{w}_n,\mathtt{v}_n))_{n\geq 0}$ is $\Omega_c^r$-valued. Therefore, repeating the arguments from the preceding proof, we get that
$$
 \int_\mathfrak{I} B(\varphi_n(s),S_n^2(\varphi)(s),\mathtt{w}_n(s),\mathtt{v}_n(s))\mbox{d}s \leq \int_\mathfrak{I} B(\varphi_0(s),S_0^2(\varphi)(s),\mathtt{w}_0(s),\mathtt{v}_0(s))\mbox{d}s
$$
for any $n\geq 0$. Consequently, by \eqref{init2} and \eqref{maj2}, we get
$$  \int_\mathfrak{I} S_n^2(\varphi)\mathtt{w}_n(s)\mbox{d}s\leq \frac{2rc}{2-r}\int_\mathfrak{I} (\varphi_n(s))^2\mathtt{w}_n(s)\mbox{d}s$$
and hence
$$  \int_\mathfrak{I} S_n^2(\varphi)w(s)\mbox{d}s\leq \frac{2rc}{2-r}\int_\mathfrak{I} (\varphi_n(s))^2w(s)\mbox{d}s.$$
If we let $n\to \infty$, the left-hand side converges to $||S(\varphi)||_{L^2_w(\mathfrak{I})}^2$, by Lebesgue's monotone convergence theorem. To deal with the right-hand side, recall that $\varphi$ belongs to $L^2_w(\mathfrak{I})$ and hence, by Muckenhoupt's inequality, so does the dyadic maximal function $\mathcal{M}_d\varphi$. Therefore, by Lebesgue's differentiation and dominated converge theorems, we see that 
$$ \int_\mathfrak{I} (\varphi_n(s))^2w(s)\mbox{d}s\xrightarrow{n\to \infty}||\varphi||_{L^2_w(\mathfrak{I})}^2.$$
This gives the claim.
\end{proof}

\section{Inequalities for continuous time martingales}\label{mart-wights}

All the results studied above have their counterparts in the martingale theory, as we will prove now. Suppose that $(\Omega,\F,\mathbb{P})$ is a complete probability space, equipped with a filtration $(\F_t)_{t\geq 0}$, i.e., a nondecreasing sequence of sub-$\sigma$-algebras of $\F$. Assume further that $\F_0$ contains all the events of probability $0$. Let $X=(X_t)_{t\geq 0}$ be an adapted, uniformly integrable  continuous-path martingale and let $\langle X \rangle=(\langle X_t\rangle)_{t\geq 0}$ stand for its quadratic covariance process (square bracket). See e.g. Dellacherie and Meyer \cite{DM} for the detailed exposition and properties of this object. Then $X$ and ${\langle X \rangle}^{1/2}$ are the probabilistic versions of the function $\varphi$ and its square function $S(\varphi)$. To introduce the appropriate analogue of $A_p$ weights, assume that $Y$ is a nonnegative, uniformly integrable martingale with continuous trajectories, satisfying $Y_0=\E[Y_{\infty}]=1$. (Note that this normalization is not an essential assumption as multiplying by a constant does not affect the $A_p$ condition.) 
 Following Izumisawa and Kazamaki \cite{IK}, we say that $Y$ satisfies Muckenhoupt's condition $A_p(mart)$ (where $1<p<\infty$ is a fixed parameter), if
 
\begin{equation}\label{probAp}
\|Y\|_{A_{p}(\text{mart})}:=\sup_{t>0}\Big\|Y_{t}\left(\E\big[\left(\frac{1}{Y_{\infty}}\right)^{1/(p-1)}\big|\calF_{t}\big]\right)^{p-1}\Big\|_{\infty}<\infty,
\end{equation}
where $Y_t=\E\left[Y_{\infty}\, |\calF_{t}\right]$. 
 Any process $Y$ as above gives rise to the probability measure $\mathbb{Q}$ defined by the equation $\mbox{d}\mathbb{Q}=Y_\infty\mbox{d}\mathbb{P}$, and thus it can be regarded as a weight. Put $Z_t=\E \big\{Y_\infty^{-1/(p-1)}\,\big| \F_t\big\}$ for $t\geq 0$. If $||Y||_{A_p}\leq c$, then we have $Y_tZ_t^{p-1}\leq c$ for all $t$, that is, the process $(Y,Z)$ takes values in the set $\Omega_c^p$. This process is precisely the martingale analogue of the sequence $((\mathtt{w}_n,\mathtt{v}_n))_{n\geq 0}$ studied in Section 3 above.

The analogue of Theorem \ref{weights} is the following.

\begin{thm}\label{contmart} 
Suppose that $Y$ is an $A_2$ weight and $X$ is a martingale bounded in $L^2(\mathbb{Q})$. Then we have the estimates
\begin{equation}\label{mainin3}
||X_\infty||_{L^2(\mathbb{Q})}\leq \large\left(80\|Y\|_{A_{2}(\text{mart})}\large\right)^{1/2}||{\langle X \rangle}^{1/2}_\infty||_{L^2(\mathbb{Q})}
\end{equation}
and
\begin{equation}\label{mainin3.5}
||{\langle X \rangle}^{1/2}_\infty||_{L^2(\mathbb{Q})}\leq 4\sqrt{2}\|Y\|_{A_{2}(\text{mart})}||X_\infty||_{L^2(\mathbb{Q})}.
\end{equation}
Furthermore, 
\begin{equation}\label{mainin4}
 ||{\langle X \rangle}^{1/2}||_{L^2(\mathbb{Q})}\leq \inf_{1<r<2}\left(\frac{r}{2-r}\|Y\|_{A_{r}(\text{mart})}\right)^{1/2}||X_\infty||_{L^2(\mathbb{Q})}.
\end{equation}
\end{thm}
Note that the constants are slightly better: this is due to the fact that we will not require Lemma \ref{geomlemma}, or any probabilistic counterpart of that statement (see the beginning of the proof below).
\begin{proof}
We will focus on \eqref{mainin3}, the reasoning leading to \eqref{mainin4} is essentially the same. Let $c=\|Y\|_{_{A_{2}(\text{mart})}}$ and let $B=B_c$ be the function of Subsection \S2.1. (Note that in contrast with the analytic setting, here we use the function $B_c$, not $B_{2c}$. This will give the aforementioned improvement of the constants). The function $B$ is of class $C^\infty$; actually, it can be extended to a $C^\infty$ function on a certain open set containing $\R\times [0,\infty)\times \Omega_c^2$. Furthermore, by the probabilistic $A_2$ condition, we see that the process $S=(X, \langle X \rangle,Y,Z)$ takes values in the domain of $B$. Thus, an application of It\^o's formula gives that for each $t\geq 0$,
\begin{equation}\label{ito}
 B(S_t)=B(S_0)+I_1+I_2+I_3/2,
\end{equation}
where
\begin{align*}
I_1&=\int_{0+}^t B_x(S_s)\mbox{d}X_s+\int_{0+}^t B_w(S_s)\mbox{d}Y_s+\int_{0+}^t B_v(S_s)\mbox{d}Z_s,\\
I_2&=\int_{0+}^t B_y(S_s)\mbox{d}\langle X \rangle_s,\\
I_3&=\int_{0+}^t D^2_{x,w,v}B(S_s)\cdot \mbox{d}\langle(X,Y,Z),(X,Y,Z)\rangle_s.
\end{align*}
Here 
$$ D^2_{x,w,v}B(x,y,w,v)=\left[ \begin{array}{ccc}
B_{xx} & B_{xw} & B_{xv}\\
B_{xw} & B_{ww} & B_{wv}\\
B_{xv} & B_{wv} & B_{vv}
\end{array}\right]$$
and the integral $I_3$ is actually a sum of the integrals
$$ \int_{0+}^t B_{xx}(S_s)\mbox{d}\langle X \rangle_s,\,\,2\int_{0+}^t B_{xw}(S_s)\mbox{d}\langle X, Y \rangle_s,\,\,2\int_{0+}^t B_{xv}(S_s)\mbox{d}\langle X,Z\rangle_s,\ldots$$
and so on. By the properties of stochastic integrals, we see that the term $I_1$ has expectation $0$. Furthermore, the sum $I_2+I_3$ is nonpositive: this follows directly from some standard approximation and the fact that the matrix $\mathbb{A}$, introduced in \eqref{D^2}, is nonpositive-definite. Consequently, integrating both sides of \eqref{ito} gives $ \E B(S_t)\leq \E B(S_0).$ However, $\langle X \rangle_0=|X_0|^2$, so \eqref{init} implies $B(S_0)\leq 0$; furthermore, by \eqref{maj}, we have 
$$ B(S_t)\geq \frac{Y_t}{2}(X_t^2-80c\langle X \rangle_t).$$
Combining these facts, we get the estimate
$$ \E X_t^2Y_t\leq 80c\E \langle X \rangle_tY_t,$$
which, by the martingale property of $Y$, implies
$$ \E X_t^2Y_\infty\leq 80c\E \langle X \rangle_tY_\infty\leq 80c\E \langle X \rangle_\infty Y_\infty.$$
This clearly implies \eqref{mainin3}, in view of Fatou's lemma. As we have mentioned above, the proofs of \eqref{mainin3.5} and \eqref{mainin4} are similar, so we leave them to the interested reader.
\end{proof}

The estimate \eqref{mainin4} leads to the following improvement of \eqref{mainin3.5}.

\begin{corollary}\label{maincor1} 
Suppose that $Y$ is an $A_2$ weight and $X$ is a martingale bounded in $L^2(\mathbb{Q})$. Then we have the estimate
\begin{equation}\label{mainin5}
||\langle X \rangle^{1/2}||_{L^2(\mathbb{Q})}\leq 2^{7/4}\|Y\|_{A_{2}(\text{mart})}||X_\infty||_{L^2(\mathbb{Q})}.
\end{equation}
\end{corollary}
\begin{proof}
We will need an appropriate probabilistic version of Coifman-Fefferman $A_{p-\e}$-result (see Lemma \ref{CFlemma} above). As shown by Uchiyama \cite{Uch}, we have the identity
\begin{equation}\label{uch}
||Y||_{A_{r}(\text{mart})}= \sup \left[\lambda^r\mathbb{Q}(\xi^*>\lambda)/||\xi_\infty||^r_{L^r(\mathbb{Q})}\right],
\end{equation}
where the supremum is taken over all $\lambda>0$ and all adapted cadlag martingales $\xi=(\xi_t)_{t\geq 0}$. On the other hand, it follows from the results of Os\c ekowski \cite{OT} that if $Y$ is an $A_2$ weight satisfying $||Y||_{A_2}=c$, then the expression on the right of \eqref{uch} does not exceed
$$ \left[\left(1-\frac{\sqrt{1-c^{-1}}}{r-1}\right)^{r-1}\left(1+\sqrt{1-c^{-1}}\right)\right]^{-1/r},$$
provided $1+\sqrt{1-c^{-1}}<r\leq 2$. As one easily verifies, the latter double bound is satisfied by $r=2-(2c)^{-1}$, and then \eqref{uch} implies
$$ ||Y||_{A_r}\leq \left[\left(1-\frac{\sqrt{1-c^{-1}}}{1-(2c)^{-1}}\right)^{1-(2c)^{-1}}\left(1+\sqrt{1-c^{-1}}\right)\right]^{-1/(2-(2c)^{-1})}.$$
However, $1+\sqrt{1-c^{-1}}\geq 1$ and
$$ 1-\frac{\sqrt{1-c^{-1}}}{1-(2c)^{-1}}=\frac{(4c^2)^{-1}}{(1-(2c)^{-1})(1-(2c)^{-1}+\sqrt{1-c^{-1}})}\geq \frac{1}{8c^2}.$$
Consequently, we obtain the upper bound
$$ ||Y||_{A_r}\leq (8c^2)^{(1-(2c)^{-1})/(2-(2c)^{-1})}\leq (8c^2)^{1/2}=2\sqrt{2}c.$$
Furthermore, we have
$$ \frac{r}{2-r}=\frac{2-(2c)^{-1}}{(2c)^{-1}}\leq 4c.$$
Plugging the above two estimates into \eqref{mainin4} gives the claim.
\end{proof}

\section{Littlewood-Paley square functions and the Lusin area integral}\label{L-P}

Our goal in  this section is to prove versions of Theorems \ref{weights} and \ref{contmart}  for the Littlewood-Paley square functions $g_{*}$ and $G_{*}$  for harmonic and parabolic functions, respectively. These operators arise  as conditional expectations of square functions of martingales obtained by composing harmonic functions in the upper half-space with Brownian motion and from martingales obtained by composing solutions of the heat equation with space-time Brownian motion.  These constructions are quite general and apply in the wide setting of general symmetric Markovian semigroups and their Poisson semigroups obtained by Bochner's 1/2-subordination.   

Before proceeding further, we mention here that there is a vast literature on weighted $L^p$ inequalities for the classical  Littlewood-Paley square functions and their many variants.  These include estimates with the sharp dependence on the characteristics $[w]_{A_p}$ of the weight.  For some of this literature we refer the reader to A. Lerner \cite{Le1, Le2, Le3} and the many references given in those papers.  Our inequalities below provide information not only on the $[w]_{A_2}$ dependence but also give $L^2$ bounds with very explicit constants.  This raises questions of obtaining sharp bounds on wighted norm inequalities for classical operators not only with respect to $[w]_{A_p}$ but also with respect to  $p$ when the weights are the probabilistic Poisson or heat (and even more general symmetric Markovian semigroup) as defined below.  Of particular interest would be the case of the Hilbert transform, first and second order Riesz transforms, and the Beurling-Ahlfors operator.  A first step in these problems would be to obtain such sharp weighted norm inequalities for martingale transforms of stochastic integrals with the probabilistic weights as defined by \eqref{probAp}.    Unfortunately,  as of now we have not been able to make progress on these problems.    

\subsection{Littlewood-Paley and Lusin square functions for harmonic functions}\label{harm-squarefRn}

In this section we will derive versions of Theorems \ref{weights} and \ref{contmart}  for the Littlewood-Paley $g_{*}$ function for harmonic and  parabolic functions.  This function dominates,  pointwise, the classical Littlewood-Paley $g$ function and the Lusin area integral. We  will use the standard construction of $g_{*}$ as the conditional expectation of the martingale square function.  For this construction we refer the reader to \cite{Ban1} which is a special case of the more general Poisson semigroup construction presented in Varopoulos \cite{Var1}.  (See also Meyer \cite{Mey1, Mey2, Mey3}.) 

For any $f\in L^p(\R^n)$, $1<p<\infty$, we will denote by $K_y(f)(x)$ its harmonic extension to the upper half-space $\R_+^{n+1}=\{ (x,y) \colon x\in\R^n,\ y > 0 \}$ obtained by convolving $f$ with the  Poisson kernel 
\begin{equation}\label{poisson}
 k_{y}(x) = \frac{c_n\, y}{(y^2+|x|^2)^{(n+1)/2}}, \,\,\,\,\,\,\,  c_n=\frac{\Gamma(\frac{n+1}{2})}{\pi^{\frac{n+1}{2}}}, 
 \end{equation}
where $c_n$ is chosen so that $k_y(x)$ has integral 1 for all $y>0$.  The cone in  $\R^{n+1}_{+}$ with vertex at $x$ and aperture $\alpha >0$ is defined by 
$$
\Gamma_{\alpha}(x)=\{(z, y): |z-x|<\alpha y, z\in \R^n, y>0\}.
$$

The   Littlewood-Paley functions $g(f)$, $g_{*}(f)$,  and  Liusin area integral $A_{\alpha}(f)$ are defined, respectively, by 
\begin{equation}\label{LP-1}
g(f)(x) = \left( \int_0^\infty  y |\nabla K_y(f)(x)|^2 \, dy\right)^{1/2}, 
\end{equation} 
\begin{equation} \label{LP-2}
g_{*}(f)(x)=\left(\int_0^{\infty}\int_{\R^d}y\,k_y(x-z) |\nabla K_y(f)(z)|^2 \, dz\, dy\right)^{1/2}
\end{equation}
and
\begin{equation} \label{LA} 
A_{\alpha}(f)(x)=\left(\int_{\Gamma_{\alpha}(x)}y^{{1-n}}|\nabla K_y(f)(x)|^2 dz dy\right)^{1/2},
\end{equation} 
where for any $u$ in the upper half-space, 
$$\nabla u= \left( \tfrac{\partial u}{\partial x_1},\
  \tfrac{\partial u}{\partial x_2}, \ldots, \tfrac{\partial u}{\partial x}, \tfrac{\partial u}{\partial y}\right)$$
is the full  gradient. 

As in the case of the dyadic square function, the inequalities \eqref{sq1} and \eqref{sq2} hold for all $f\in L^p(\R^n)$, for $1<p<\infty$, for both  $g$ and $A_{\alpha}$.  For $g_{*}$  we have  $\|g_{*}(f)\|_p\leq c_p\|f\|_p$ for $2\leq p<\infty$ and it is well-known that the inequality fails for $1<p<2$; see \cite{Ben} for an explicit example.   We refer the reader to Stein \cite{Ste1} for these classical results and where it is also shown that there are constants $C_{\alpha, n}^1$, $C_{\alpha, n}^2$, depending only on $\alpha$ and $n$,  such that 
\begin{equation}\label{pointwise}
g(f)(x)\leq C_{\alpha, n}^1A_{\alpha}(f)(x)\leq C_{\alpha, n}^2g_{*}(f)(x), \, \, \text{for all}\,\, \, x\in \R^n. 
\end{equation}
In fact, the second inequality is trivial since 
$$y^{-n}\leq \frac{(\alpha^2+1)^{\frac{n+1}{2}}}{c_n}k_y(z-x), \,\,\,  (z, y)\in \Gamma_{\alpha}(x),$$
where $c_n$  is the constant in \eqref{poisson}.  This gives 
\begin{equation}\label{pointwiseAbyg*} 
A_{\alpha}(f)(x)\leq \frac{(\alpha^2+1)^{\frac{n+1}{4}}}{\sqrt{c_n}}g_{*}(f)(x). 
\end{equation} 
Also, the semigroup property of $K_y$ gives directly (see also the proof of \eqref{pointwiseGbyG*} below) that 
\begin{equation}\label{pointwisegbyg*} 
g(f)(x)\leq 2g_{*}f(x).  
\end{equation} 

 An important property  for our purpose in this paper  is the fact that $g_{*}(f)$ can be expressed as the  conditional expectation of the quadratic variation (square function) of the martingale obtained by composing the harmonic function $K_yf(x)$ with the Brownian motion in $\R^{n+1}_{+}$.  Let us explain this further. Let $B_t= (X_t, Y_t)$, $t\geq 0$, be Brownian motion  in $\R^{n+1}_{+}$ starting at the point $(z, a)$ so that  $(X_t)_{t\geq 0}$ is an $n$-dimensional Brownian motion and $(Y_t)_{t\geq 0}$ is a one-dimensional Brownian motion.  We denote the corresponding starting probability measure and expectation by $\p_{(z, a)}$ and $\E_{(z, a)}$, respectively.    Let $\tau$ be its exit time from $\R^{n+1}_{+}$ so that $\tau=\inf\{s>0: Y_s=0\}$. Since the density of the distribution of the random variable $X_\tau$ under the probability measure $\p_{(z, a)}$ is given by the Poisson kernel $k_a(z-x)$, Fubini's theorem gives that for nonnegative (or integrable) functions $F$ on $\R^n$ we have 

\begin{equation}\label{uniform}
\int_{\R^n}\E_{(z, a)}F(X_\tau) dz=\int_{\R^n} \left(\int_{\R^n} k_a(z-x)F(x) dx\right)dz=\int_{\R^n} F(x) dx, 
\end{equation}
for all $a>0$.  This simple formula is used below multiple times to convert integrals over $\R^n$ with respect to the Lebesgue measure to expectations which then permits the application of martingale inequalities.  
 
We now consider the martingale $M(f)_t=K_{Y_{ \tau\wedge t}}(f)(X_{\tau\wedge t})$, $t\geq 0$. By the It\^o formula,  
$$
M(f)_t=K_af(z)+\int_{0}^{\tau\wedge t} \nabla K_{Y_s}(f)(X_s)\cdot dB_s
$$
and its quadratic variation of the martingale is given by 
$$
\langle M(f)\rangle_{t}=|K_af(z)|^2+\int_{0}^{\tau\wedge t} |\nabla K_{Y_s}(f)(X_s)|^2\,ds. 
$$
Setting 
$$
\E_{(z, a)}^{x}\left(\int_{0}^{\tau}  |\nabla K_{Y_s}(f)(X_s)|^2 ds\right)=\E_{(z, a)}\left(\int_{0}^{\tau} 
 |\nabla K_{Y_s}(f)(X_s)|^2ds\, \big|B_{\tau}=(x, 0)\right), 
$$
it is proved in  Ba\~nuelos \cite[p. 663]{Ban1}) that 
\begin{eqnarray}\label{g*harm1}
g_{*, a}^2(f)(x)&=&\int_0^{a}\int_{\R^n}y\,k_y(x-z) |\nabla K_y(f)(z)|^2 \, dz\, dy\\
&=&\frac{1}{2}\int_{\R^n}\E_{(z, a)}^{x}\left(\int_{0}^{\tau}  |\nabla K_{Y_s}(f)(X_s)|^2\,ds\right)k_{a}(x-z) dz.\nonumber 
\end{eqnarray}

\subsection{Littlewood-Paley and Lusin square functions for parabolic functions}\label{heat-squarefRn}
As it is well known, the Poisson kernel in the above construction can be replaced by more general volume preserving dilations of approximations to the identity and these give rise to more general Littlewood-Paley and Lusin square functions that have similar $L^p$ boundedness properties. See for example, \cite{BanMoo}.  A probabilistic way (also well known by now) to generalize the Poisson kernel construction above is to replace it with a more general Poisson semigroup obtained by the Bochner 1/2-subordination of a symmetric Markovian semigroup.   This is what is done in Varopoulos \cite{Var1} and Meyer \cite{Mey1}.  A variation of this construction applies to space-time martingales arising from the Markovian semigroup itself and not just by its 1/2-subordination.  This construction was used, in for example, \cite{BanMen}, for applications to the Beurling-Ahlfors operator and second order Riesz transforms.    For our purpose here, and to connect to the classical $A_p$-weights and the classical parabolic square functions in analysis, we present the construction for the heat (Gaussian kernel) semigroup on  $\R^n$.  
  For $t>0$, $x\in \R^n$, denote the heat (Gaussian)  kennel by 
\begin{equation}\label{gauss-ker}
 p_t(x)=\frac{1}{(2\pi t)^{\frac{n}{2}}}e^{-\frac{|x|^2}{2t}}
\end{equation}
 and this time consider the heat extension $P_{t}f(x)=(p_t{*}f)(x)$ of $f$.   The parabolic ``cone" in
$\R_+^{n+1}$ with vertex at $x$
and aperture $\alpha$ is defined by
 $$
\calP_\alpha(x)=\{(z,t)\in \R_+^{n+1}:|z-x| < \alpha \sqrt{t}\}. 
$$ 
 
 The  corresponding parabolic Littlewood-Paley functions $G(f)$, $G_{*}(f)$ and parabolic Liusin area function $\calP A_\alpha f(x)$ 
 are given, receptively,  by   
 \begin{equation}\label{G-heat}
 G(f)(x)=\left(\int_0^{\infty}|\nabla_{x}P_{t}f(x)|^2\, dt\right)^{1/2},
\end{equation}
\begin{equation}\label{G*-heat}
G_{*}(f)(x)=\left(\int_0^\infty \int_{\R^d} p_t(x-z)|\nabla_xP_{t}f(z)|^2\,dzdt\right)^{1/2}, 
\end{equation}
and 
\begin{equation}\label{Lus-heat}
\calP A_\alpha (f)(x)=\left( \int_{\calP_\alpha(x)} t^{-{n\over
2}}|\nabla_x P_{t}f(z)|^2 dz
dt\right)^{1/2},
\end{equation} 
where for any function $U(x, t)$ in the upper half-space 
$$\nabla_xU = \left( \tfrac{\partial U}{\partial x_1},\ \tfrac{\partial U}{\partial x_2}
  \ldots, \tfrac{\partial U}{\partial x_n} \right)$$  
  denotes its ``horizontal"  gradient.  
  
  These square functions have also been widely studied  in the literature.  We refer the reader to \cite{BanMoo} (and references given there) for some of their basic properties.  As in the case of harmonic functions we have $\|G(f)(x)\|_p\approx \|\calP A_\alpha f(x)\|_p\approx \|f\|_p$, for $1<p<\infty$,  and $\|G_{*}(f)\|_p\leq c_p\|f\|_p$, $2\leq p<\infty$.    Similarly, the the following pointwise inequality holds: 
 \begin{equation}\label{heat-pointwise}
G(f)(x)\leq C_{\alpha, n}^3\calP A_\alpha (f)(x)\leq C_{\alpha, n}^4 G_{*}(f)(x), \, \, \text{for all}\,\, \, x\in \R^n, 
\end{equation}
for some constants $C_{\alpha, n}^3$ and $C_{\alpha, n}^4$ depending only on $\alpha$ and $n$.  In fact, since 
$$t^{-n/2}\leq (2\pi)^{n/2}e^{\alpha^2/2}p_t(x-z), \,\,\,\, \text{for}\,\,\,\,   z\in \calP_\alpha(x),$$
we have 
\begin{equation}\label{pointwiseAbyG*}
\calP A_\alpha (f)(x)\leq (2\pi)^{n/4}e^{\alpha^2/4} G_{*}(f)(x)
\end{equation} 
Similarly,  by the semigroup property, 
$$
P_{t}f(x)=P_{t/2}(P_{t/2}f)(x)=\int_{\R^n} p_{t/2}(x-z)P_{t/2}f(z)dz
$$
and by  Jensen's inequality,
$$
|\nabla_{x} P_{t}f)(x)|^2\le \int_{\R^n} p_{t/2}(x-z)|\nabla_x P_{t/2}f(z)|^2dz.  
$$
This gives 
\begin{eqnarray}\label{pointwiseGbyG*}
G^2(f)(x)&=&\int_0^{\infty} |\nabla_x P_{t}f(z)|^2dt\nonumber\\
&\le &\int_0^{\infty}\int_{\R^n} p_{t/2}(z-x)|\nabla_x P_{t/2}f(z)|^2dzdt\\
&=&2G^2_{*}(f)(x)\nonumber. 
\end{eqnarray}
  
  As in the case of the Poisson kernel, the Littlewood-Paley $G_{*}$ function is the conditional expectation of the corresponding martingale square function.  This time, however, the martingales arise from space-time Brownian motion.  This fact is proved exactly as in \cite{Var1} or \cite[p. 663]{Ban1},  once the space-time martingale is identified.  For the sake of completeness, we briefly explain this. For the space-time martingale details as used here, see for example \cite{BanMen}. 
  
Let $(B_t)_{t\geq 0}$ be the Brownian motion in $\R^n$ starting at $z$ and let $\p_{z}$ and $\E_{z}$ be the probability  and expectation for $B$.  Fix  $0<T<\infty$.  Then 
$$
M(f)_{t} =P_{T-t}f(B_t),\qquad 0\le t\le T, 
$$
is a martingale and by the  It\^o's formula, 
$$
M(f)_{t}= P_{T}f(z)+ \int_0^t\nabla_x P_{T-s}f(B_s)\cdot dB_s. 
$$
The quadratic variation (square function) of this martingale is given by
$$
\langle M(f)\rangle_t=|P_{T}f(z)|^2+ \int_0^t|\nabla_x P_{T-s}f(B_s)|^2\,ds. 
$$
Setting 
\begin{equation}\label{G*-heat-truncated}
G_{*, T}(f)(x)=\left(\int_0^T \int_{\R^d} p_t(x-z)|\nabla_x P_{t}f(z)|^2 dz\,dt\right)^{1/2}, 
\end{equation}
we claim that  
\begin{equation}\label{Cond-Expect1}
G^2_{*, T}f(x)= \int_{\R^n} \E_{z}^{x}\left(\int_{0}^{T}|\nabla_x P_{T-s}f(B_s)|^2\,ds\right)p_{T}(x-z) dz,
\end{equation}
where
$$
\E_{z}^{x}\left(\int_{0}^{T} |\nabla_x P_{T-s}f(B_s)|^2\,ds\right)=\E_{z}\left(\int_0^T|\nabla_x P_{T-s}f(B_s)|^2\,ds\,\big|\,B_T=x\right).
$$
To  prove \eqref{Cond-Expect1} recall that the conditional distribution of $B_s$ given 
$B_T=x$ (the Brownian bridge with $B_0=z$, $B_T=x$) is 
$$\frac{p_s(z-w)p_{T-s}(w-x)}{p_T(z-x)}\,dw. 
$$
 Thus, 
\begin{align*}
&\E_{z}^{x}\left(\int_{0}^{T}|\nabla_z P_{T-s}f(B_s)|^2ds\right)\\
&=\int_0^T\int_{\R^n} \frac{p_s(z-w)p_{T-s}(w-x)}{p_T(z-x)}|\nabla_x P_{T-s}f(w)|^2dw ds.
\end{align*}
Multiplying both sides of this equality by $p_T(z-x)$  and integrating on $z$ we obtain 
\begin{eqnarray*}
&&\int_{\R^d} \E_{z}^{x}\left(\int_{0}^{T}|\nabla_x P_{T-s}f(B_s)|^2ds\right)p_{T}(z-x) dz\\
&=&
\int_0^T\int_\Rd \int_{\R^n}p_s(z-w)p_{T-s}(w-x)|\nabla_xU_f(w, T-s)|^2\,dwdzds\\
&=& \int_0^T\int_{\R^d}p_{T-s}(y-z)|\nabla P_{T-s}f(y)|^2\,dy,ds\\
&=&\int_0^T \int_{\R^d}p_{s}(w-x)|\nabla_xP_{T-s}f(w)|^2\,dw\,ds=G^2_{*, T}f(x), 
\end{eqnarray*}
which verifies \eqref{Cond-Expect1}. 

\subsection{Poisson $A_p$ weights in the Disc} We now explore the connections between the martingale $A_p$ weights studied in \S \ref{mart-wights} and various classes previously studied in analysis.  This connections are more transparent for the unit disk in the plane (or unit ball in $\R^n$) where the Brownian motion has a natural place to start, namely the origin.  For further clarity and to connect to the classical Littlewood-Paley square functions, we first treat this case.  Let $\D=\{x\in \bC: |z|<1\}$ be the unit disc in the complex plain with the circle $\T=\partial D$ as its boundary and Poisson kernel  given by 
$$P_z(e^{i\theta})=\frac{1-|z|^2}{|z-e^{i\theta}|^2}, \,\,\,\, z\in \D.$$
For the rest of this section we assume that $w$ is a positive and integrable function on the unit circle $\T$.  Let 
$$
u_{w}(z)=\frac{1}{2\pi}\int_{\T} P_z(e^{i\theta}) w(e^{i\theta}) d\theta
$$ 
be the Poisson integral of $w$.  
\begin{definition}
 We say that $w\in A_p(\text{Poisson},\T)$ if 
\begin{equation}\label{Poisson-T}
\|w\|_{A_{p, \T}}=\|u_{w}(z)\left(u_{w^{-1/p-1}}(z)\right)^{p-1}\|_{L^{\infty}(\D)}<\infty. 
\end{equation}
\end{definition}

Now, let $(B_t)_{t\geq 0}$ be Brownian motion in $\D$ starting at the origin and let $\tau$ be its first exit time from $\D$.  Since $u_w(z)$ is harmonic, the process $Y_t=u_{w}(B_{\tau\wedge t})$, $t\geq 0$, is a martingale with $Y_{\infty}=w(B_\tau)$. By the strong Markov property, 
$$\E_{0}\left(Y_{\infty}\,\,\big| \F_{\tau\wedge t}\right)=\E_{B_{\tau\wedge t}}\left(w(B_\tau)\right)=Y_t$$ 
and similarly for $Y_{\infty}^{-1/(p-1)}$:
$$
\E_0\left( Y_{\infty}^{-1/(p-1)}\big| \F_{\tau\wedge t}\right)=  \E_0\left(u_{w^{-1/p-1}}(B_{\tau})\,\big| \F_{\tau\wedge t}\right)=
\E_{B_{\tau\wedge t}}\left(u_{w^{-1/p-1}}(B_{\tau})\right).
$$
Recalling the martingale $A_p$ weights defined in  \eqref{probAp}, we see that $Y\in A_p(\text{mart})$ if and only if $w\in A_p(\text{Poisson}, \T)$ and in fact, we have 
\begin{equation}\label{mart&T-equiv} 
\|Y\|_{A_p(\text{mart})}=\|w\|_{A_{p, \T}}. 
\end{equation}

\begin{remark}
The $A_p(\text{Poisson},\T)$ weights have been studied in recent years in connection with the $L_w^2(\T)$ boundedness of the conjugate function (Hilbert transform) with the correct dependence on the constant $\|w\|_{A_{p, \T}}$.  For this, we refer the reader to S. Petermichl and J. Wittwer \cite{PetWit1}.  The fact that these weights are probabilistic weights for the corresponding martingales has been known for many years. The first author learned this from R. Durrett  in the early 1980's. 
\end{remark}  

The Littlewood-Paley  $g_{*}$ function on the circle $\T$ is defined by  
\begin{equation}\label{doob-g_{*}} 
g_*(f)(e^{i\theta})=\left( {1\over \pi} \int_{\D} {1-|z|^2\over
|z-e^{i\theta}|^2}\log {1\over
|z|} |\nabla u_f(z)|^2 dz\right)^{1/2},   
\end{equation} 
where $dz$ denotes the area measure in the plane.  This version, which is pointwise comparable to the classical  Zygmund \cite{Zyg} $g_{*}$ function, was introduced in \cite{Ban1}.  
As in the case of $\R^n$, this  square function is the conditional expectation of the square function of the martingale $u_f(B_{\tau\wedge t})$.  That is, by the It\^o formula, 
we have
$$
u_f(B_{\tau\wedge t})=u_f(0)+\int_0^{\tau\wedge t}\nabla u_f(B_s)\cdot dB_s
$$  
for all $t$, and hence this martingale has the square function given by
$$
\langle u_f(B)\rangle_{\tau\wedge t} =|u_f(0)|^2+\int_0^{\tau\wedge t}|\nabla u_f(B_s)|^2ds. 
$$
Now, we have
$$
g^2_*(f)(e^{i\theta})=\E_{0}^{\theta}\left(\int_0^{\tau}|\nabla u_f(B_s)|^2 ds\right)=\E_{0}\left( \int_0^{\tau}|\nabla u_f(B_s)|^2 ds \,\,\big| B_\tau=e^{i\theta} \right).
$$
We refer the reader to \cite[p.~650]{Ban1} for the details on this formula which is proved using the transition probabilities for the Doob $h$-process for Brownian motion starting at $0$, conditioned to exit $\D$ at $e^{i\theta}$. Since $B_{\tau}$  is uniformly distributed on $\T$ under $\p_0$, we have 
\begin{eqnarray}\label{g-disc}
\frac{1}{2\pi}\int_{\T} g^2_*(f)(e^{i\theta})w(e^{i\theta})d\theta&=&\E_0\left(\E_0\left( \int_0^{\tau}|\nabla u_f(B_s)|^2 ds\,\,\big| B_\tau\right)w(B_\tau)\right)\nonumber\\
&=&\E_0\left( \left(\int_0^{\tau}|\nabla u_f(B_s)|^2 ds\right)\,w(B_\tau)\right). 
\end{eqnarray}

\begin{theorem}\label{A_p(T)} Suppose $w\in A_2(Poisson, \T)$ and  $f\in C(\T)$, the space of continuous functions in $\T$. Then, 
\begin{equation}\label{mainnT1}
||f-u_f(0)||_{L^2_w(\T)}\leq \large\left(80\|w\|_{A_{2, \T}}\large\right)^{1/2}||g_*(f)(e^{i\theta})||_{L^2_w(\T)},
\end{equation}
\begin{equation}\label{maininT2}
 \|g_*(f)(e^{i\theta})\|_{ L_w^2(\T)}\leq \inf_{1<r<2}\left(\frac{r}{2-r}\|w\|_{A_{r, \T}}\right)^{1/2}\|f\|_{ L_w^2(\T)}
\end{equation}
and 
\begin{equation}\label{maininT3}
\|g_*(f)(e^{i\theta})\|_{ L_w^2(\T)}\leq 2^{7/4}\|w\|_{A_{2, \T}}\|f\|_{ L_w^2(\T)}, 
\end{equation}
where 
$u_f(0)=\frac{1}{2\pi}\int_{T} f(e^{i\theta}) d\theta$.  
\end{theorem}  

\begin{proof}  Applying \eqref{mart&T-equiv} and the inequality \eqref{mainin3} of Theorem \ref{contmart}, we have 
\begin{eqnarray*} 
\frac{1}{2\pi}\int_{\T}|f(e^{i\theta})-u_f(0)|^2 w(e^{i\theta})d\theta&=&\E_0\left(|f(B_{\tau})-u_f(0)|^2\,w(B_{\tau}\right)\\
&=&E_0\left(\big|\int_0^{\tau}\nabla u_f(B_s)dB_s\big|^2w(B_\tau)\right)\\
&\leq &80\|w\|_{A_{2, \T}}\E_0\left( \left(\int_0^{\tau}|\nabla u_f(B_s)|^2 ds\right)\,w(B_\tau)\right)\\
&=&80\|w\|_{A_{2, \T}}\left(\frac{1}{2\pi}\int_{\T} g^2_*(f)(e^{i\theta})w(e^{i\theta})d\theta\right),
\end{eqnarray*}
where the last equality follows from \eqref{g-disc}.  This proves \eqref{mainnT1}.  
To establish \eqref{maininT2} we apply \eqref{mainin4} to obtain 
\begin{eqnarray*} 
\frac{1}{2\pi}\int_{\T} g^2_*(f)(e^{i\theta})w(e^{i\theta})d\theta&=&\E_0\left( \left(\int_0^{\tau}|\nabla u_f(B_s)|^2 ds\right)\,w(B_\tau)\right)\\
&\leq & \E_0\left( \left(|u_f(0)|^2+\int_0^{\tau}|\nabla u_f(B_s)|^2 ds\right)\,w(B_\tau)\right)\\
&\leq & \inf_{1<r<2}\left(\frac{r}{2-r}\|w\|_{A_{r, \T}}\right)\,\E_{0}\left(|f(B_{\tau})|^2\, w(B_{\tau})\right)\\
&=& \inf_{1<r<2}\left(\frac{r}{2-r}\|w\|_{A_{r, \T}}\right)\left(\frac{1}{2\pi}\int_{\T}|f(e^{i\theta})|^2 w(e^{i\theta})d\theta\right). 
\end{eqnarray*} 
Finally, \eqref{maininT3} is proved the same way applying Corollary \ref{maincor1}.  
\end{proof} 

For $0<\alpha<1$ the Stoltz domain, denoted here by $\Gamma_{\alpha}(\theta)$, is the interior of the smallest convex set containing the disc $\{z\in \bC: |z|<\alpha\}$ and the point $e^{i\theta}$.  The Lusin area function (area integral) of $f$ is 
$$
A_{\alpha}(f)(e^{i\theta})=\left(\int_{\Gamma_{\alpha}(\theta)} |\nabla u_f(z)|^2 dz\right)^{1/2}. 
$$
Similarly, the Littlewood-Paley function $g$ is defined by 
$$
g(f)(e^{i\theta})=\left(\int_0^1 (1-r)|\nabla u_f(re^{i\theta})|^2 dr\right)^{1/2}.
$$
As before, it is  easy to show that there are universal constant $C_{\alpha}$ and $C$ such that the pointwise inequalities $A_{\alpha}(f)(e^{i\theta})\leq C_{\alpha} g_*(f)(e^{i\theta})$ and $g(f)(e^{i\theta})\leq C g_*(f)(e^{i\theta})$ hold.   This gives the following Corollary.  

\begin{corollary} Suppose $w\in A_2(Poisson, \T)$ and  $f\in C(\T)$. Then
\begin{equation}
\|A_{\alpha}(f)(e^{i\theta})\|_{L^2(w)}\leq 2^{7/4}C_{\alpha}\|w\|_{A_{2, \T}}\|f\|_{L^2(w)}
\end{equation}
and 
\begin{equation}\label{maininT4}
\|g(f)(e^{i\theta})\|_{L^2(w)}\leq 2^{7/4}C\|w\|_{A_{2, \T}}\|f\|_{L^2(w)}.
\end{equation}

\end{corollary} 

\subsection{Poisson and heat $A_p$ weights on $\R^n$} In this section we carry out the computations  in $\R^n$ done above for the disc.   We follow the notation of \S \ref{harm-squarefRn}. 

\begin{definition}\label{heatPoisson} Let $w$ be a positive locally integrable function defined on $\R^n$ and $1<p<\infty$.    
\item[(i)] We will say $w\in A_p(\text{Poisson},\R^n)$ if 
\begin{equation}\label{A_p-harmonic}
\|w\|_{A_p(\text{Poisson}, \R^n)}:=\|K_y{w}(x) \left(K_y({w^{-1/(p-1}})(x)\right)^{p-1}\|_{L^{\infty}(\R^{n+1}_{+})}<\infty.
\end{equation}
\item[(ii)] 
We will say $w\in A_p(\text{heat}, \R^n)$ if 
\begin{equation}\label{A_p-heat}
\|w\|_{A_p(\text{heat}, \R^n)}:=\|P_tw(x) \left(P_t({w^{-1/(p-1}})(x)\right)^{p-1}\|_{L^{\infty}(\R^{n+1}_{+})}<\infty. 
\end{equation}
\end{definition}

We remark here that  these $A_p$ weights can be defined for any Markovian semigroup and not just for the Poisson or heat semigroup in $\R^n$. Such $A_p$ weights are nothing more than the martingale ${A_{p}(\text{mart})}$ weights arising from the stochastic process associated with the semigroup.  Before we explain this more precisely, we recall that both classes of $A_p$ weights defined above have been studied before in connection to weight problems and applications.  Indeed, it was proved by S. Petermichl and A. Volberg in \cite{PetVol1} that there are constants $a$ and $b$, depending only on the dimension $n$, such that 
\begin{equation}\label{PetVol}
aA_p{(\text{heat}, \R^n)}\leq \|w\|_{A_p}\leq b A_p{(\text{heat}, \R^n)}, 
\end{equation}
where 
$\|w\|_{A_p}$ is as in the original definition of Muckenhoupt.  That is, $w\in A_p$ if 
$$ \|w\|_{A_p}=\sup_Q \left(\frac{1}{|Q|}\int_Q w \right)\left(\frac{1}{|Q|}\int_Q w^{-1/(p-1)}\right)^{p-1}<\infty,$$
where the sup is taken over all cubes $Q\subset \R^n$.  As for $A_p(\text{Poisson},\R^n)$, it is known that when $n=1$, $A_2(\text{Poisson})=A_2$ and that in fact there are universal constants $a$ and $b$ such that 
\begin{equation}\label{Pet1}
a\|w\|_{A_2}\leq \|w\|_{A_2(\text{Poisson}, \R^n)}\leq b\|w\|_{A_2}^2.
\end{equation}
On the other hand, for $n>1$ there are weights $w$ for which we have $\|w\|_{A_2}<\infty$, but $\|w\|_{A_2(\text{Poisson}, \R^n)}=\infty$.  Thus, for $n>1$, $A_2\neq A_2(\text{Poisson}, \R^n)$.  For these results, as well as the boundedness of the classical Riesz transforms on $L_w^2(\R^n)$, $w\in  A_2(\text{Poisson})$  with constants independent of the dimension $n$, we refer the reader to Hukov\'ic \cite{HTV}, Petermichl \cite{PetWit1} and K. Domelevo, Petermichl and  Wittwer \cite{DemPetWit}. 

\begin{remark}\label{A_pbyApoisson/heat}  
We remark here that while $A_2\neq A_2(\text{Poisson}, \R^n)$,  it is  easy to see that $A_2(\text{Poisson}, \R^n)\subset A_2$ for all $n\geq 1$.  Indeed, given a cube $Q$ centered at $z$ and length $\l_{Q}$, we pick $a\approx \l_{Q}$ to obtain that $\frac{1}{|Q|}\chi_{Q}\leq C_nk_a(z-x)$ for all $x\in Q$ for some universal constant $C_n$ depending only on $n$. This imediatly shows that $A_2(\text{Poisson}, \R^n)\subset A_p$ for all $1<p<\infty$ and $n\geq 1$.  The same argument (picking this time $t^2\approx \l_Q$) shows that $\frac{1}{|Q|}\chi_{Q}\leq C_n' p_t(z-x)$ for all $x\in Q$. This gives that $A_p(\text{heat},\R^n)\subset A_p$ for all $1<p<\infty$ and $n\geq 1$.
\end{remark} 

Our aim now is to prove versions of Theorem  \ref{contmart} and Corollary \ref{maincor1} for the Littlewood-Paley functions $g_{*}$ and $G_{*}$ 
with respect to weights in $A_2(\text{Poisson}, \R^n)$ and $A_p(\text{heat})$.  

\begin{lemma}\label{equiPoisson1} Suppose $w\in A_p(\text{Poisson}, \R^n)$.  Fix $(z, a)\in \R^{n+1}_{+}$ and let $B_t=(X_t, Y_i)$, $t\geq 0$, be Brownian motion $\R^{n+1}_{+}$ starting at $(z, a)$ and denote by $\tau$ its exit time from $ \R^{n+1}_{+}$.  Let  $\tilde{Y}_t=K_{Y_{\tau\wedge t}}w(X_{\tau\wedge t})$, $t\geq 0$, be the martingale under the measure $\p_{(z, a)}$. Then $\tilde{Y}_{\infty}=w(X_{\tau})\in A_p(\text{mart})$ and $\|\tilde Y_{\infty}\|_{A_{p}(\text{mart})}\leq \|w\|_{A_p(\text{Poisson},  \R^n)}$
\end{lemma}
\begin{proof} 
With  $\tilde{Y}_{\infty}=w(X_{\tau})$ and  $K_{a}w(z)=\E_{(z, a)}w(X_{\tau})$, the Strong Markov property gives 
\begin{eqnarray*}
\E_{(z,a)}\left(\left(\frac{1}{\tilde{Y}_{\infty}}\right)^{1/(p-1)}\big|\F_{\tau\wedge t}\right)&=& \E_{(z,a)}\left(\left(\frac{1}{w(X_{\tau})}\right)^{1/(p-1)}\big|\F_{\tau\wedge t}\right)\\
&=& \E_{B_{\tau\wedge t}}\left(\left(\frac{1}{w(X_{\tau})}\right)^{1/(p-1)}\right)\\
&=&K_{Y_{\tau\wedge t}}({w^{-1/(p-1}})(X_{\tau\wedge t}).
\end{eqnarray*} 
Hence, 
\begin{align*}
&\tilde{Y}_t\, \left(\E_{(z,a)}\left(\left(\frac{1}{\tilde{Y}_{\infty}}\right)^{1/(p-1)}\big|\F_{\tau\wedge t}\right)\right)^{p-1}\\
&\qquad \qquad =
K_{Y_{\tau\wedge t}}w(X_{\tau\wedge t})\left(K_{Y_{\tau\wedge t}}({w^{-1/(p-1}})(X_{\tau\wedge t})\right)^{p-1}.
\end{align*}
It follows from this that  for all $(z, a)\in \R^{n+1}_{+}$, 
\begin{eqnarray*}\label{equiv2}
\|\tilde Y_{\infty}\|_{A_{p}(\text{mart})}&=& \sup_{t\geq 0}\Big\|\tilde{Y}_{t}\left(\E_{(z,a)}\left(\left(\frac{1}{\tilde{Y}_{\infty}}\right)^{1/(p-1)}\big|\F_{\tau\wedge t}\right)\right)^{p-1}\Big\|_{\infty}\\
&=&\sup_{t>0}\|K_{Y_{\tau\wedge t}}w(X_{\tau\wedge t})\left(K_{Y_{\tau\wedge t}}({w^{-1/(p-1}})(X_{\tau\wedge t})\right)^{p-1}\|_{\infty}\\&\leq& \|w\|_{A_p(\text{Poisson})}.
\end{eqnarray*}
This completes the proof. 
\end{proof} 
 
 \begin{remark}\label{remark1} It is important to note here, for our applications below,  that the above inequality $\|\tilde Y\|_{A_{p}(\text{mart})}\leq \|w\|_{A_2(\text{Poisson},  \R^n)}$ holds for all starting points $(z, a)$. While not needed for the purpose of this paper, we note that here we actually have equality.  That is,  $\|\tilde Y\|_{A_{p}(\text{mart})}=\|w\|_{A_2(\text{Poisson},  \R^n)}$.  This follows from the fact that if  $F(x, y)$ is a continuous bounded function in the upper half-space then 
 $$\sup_{t}\|F(X_{\tau\wedge t}, Y_{\tau\wedge t})\|_{L^{\infty}(\p_{(z, a)})}=\|F(x, y)\|_{L^{\infty}(\R^{n+1}_{+})},$$
 since given any ball $B$ in the upper-half space there will be a time $t>0$ such that 
 $\p_{(z, a)}\{\left(X_{\tau\wedge t}, Y_{\tau\wedge t})\in B\right)\}>0$.  Indeed, this quantity is given by the integral of the Dirichlet heat kernel in the upper half-space (which is just the product of the heat kernel in $\R^n$ and heat kernel for the half line) over the ball $B$.  
\end{remark}

\begin{theorem}\label{A_2(Poisson-Rn)}  Suppose $w\in A_2(\text{Poisson},  \R^n)$ and $f\in C_{0}(\R^n)$, the space of continuous functions of compact support.   Then, 
\begin{equation}\label{mainnRn1}
||f||_{L_w^2(\R^n)}\leq \large\left(320\|w\|_{A_2(\text{Poisson},  \R^n)}\large\right)^{1/2}||g_*(f)||_{L_w^2(\R^n)},
\end{equation}

\begin{equation}\label{maininRn2}
||g_{*}(f)||_{L_w^2(\R^n)}\leq \frac{1}{\sqrt{2}} \inf_{1<r<2}\left(\frac{r}{2-r}\|w\|_{A_2(\text{Poisson},  \R^n)}\right)^{1/2}\|f\|_{L_w^2(\R^n)}
\end{equation}
and 
\begin{equation}\label{maininRn3}
\|g_*(f)\|_{L_w^2(\R^n)}\leq 2^{5/4}\|w\|_{A_2(\text{Poisson},  \R^n)}\|f\|_{L_w^2(\R^n)}.
\end{equation}
\end{theorem}
\begin{proof}

Let $R$ be large enough so that the support of $f$ is contained in the ball $B=B(0, R)$.      
By \eqref{uniform} we have 
\begin{eqnarray}\label{term1}
\int_{\R^n}|f(x)|^2 w(x) dx&=&\int_{\R^n} \mathbf{1}_{B}(x)|f(x)|^2 w(x) dx\nonumber\\
&=&\int_{\R^n} E_{(z, a)}\left(\mathbf{1}_{B}(X_{\tau})|f(X_{\tau})|^2 w(X_{\tau})\right)dz\nonumber\\
&\leq &2\int_{\R^n} E_{(z, a)}\left(\mathbf{1}_{B}(X_{\tau})|f(X_{\tau})-K_af(z)|^2 w(X_{\tau})\right)dz\\
&+&2\int_{\R^n}|K_af(z)|^2\E_{(z, a)} \mathbf{1}_{B}(X_{\tau})w(X_{\tau}) dz. \nonumber
\end{eqnarray} 
We now estimate the first term under the integral on the right hand side of the above inequality.  By 
Theorem \ref{contmart} and Lemma \ref{equiPoisson1} we have 
\begin{eqnarray*}
&&2\E_{(z, a)}\left(|f(X_{\tau})-K_af(z)|^2 w(X_{\tau})\right)\\
&\leq & 160 \|w\|_{A_2(\text{Poisson})}\E_{(z, a)}\left(\int_{0}^{\tau} |\nabla K_{Y_s}(f)(X_s)|^2\,ds\, w(X_{\tau})\right)\\
&=& 160 \|w\|_{A_2(\text{Poisson},  \R^n)}\E_{(z, a)}\left(\E_{(z, a)}\left(\int_{0}^{\tau} |\nabla K_{Y_s}(f)(X_s)|^2\,ds\, w(X_{\tau})\big| X_{\tau}\right) \right)\\
&=& 160 \|w\|_{A_2(\text{Poisson},  \R^n)}\E_{(z, a)}\left(\E_{(z, a)}\left(\int_{0}^{\tau} |\nabla K_{Y_s}(f)(X_s)|^2\,ds\,)\big| X_{\tau}\right) w(X_{\tau}\right)\\
&=& 160 \|w\|_{A_2(\text{Poisson},  \R^n)}\E_{(z, a)}\left(\E_{(z, a)}^{X_{\tau}}\left(\int_{0}^{\tau} |\nabla K_{Y_s}(f)(X_s)|^2\,ds \right) w(X_{\tau}\right)\\
&=& 320 \|w\|_{A_2(\text{Poisson},  \R^n)}\E_{(z, a)}\left[g^2_{*, a}(f)(X_{\tau}) w(X_{\tau})\right].
\end{eqnarray*}
Integrating both sides of this inequality in $z$ gives 
\begin{eqnarray*}\label{term2}
&&2\int_{\R^n}\E_{(z, a)}\left(|f(X_{\tau})-K_af(z)|^2 w(X_{\tau})\right) dz\\
&\leq&320 \|w\|_{A_2(\text{Poisson},  \R^n)}\int_{\R^n}\E_{(z, a)}\left[g^2_{*, a}(f)(X_{\tau}) w(X_{\tau})\right] dz\\
&=& 320 \|w\|_{A_2(\text{Poisson},  \R^n)}\int_{\R^n}g^2_{*, a}(f)(x) w(x) dx\\
&\leq&320 \|w\|_{A_2(\text{Poisson},  \R^n)} \int_{\R^n}g^2_{*}(f)(x) w(x) dx.
\end{eqnarray*} 
Combining this with \eqref{term1} we obtain
\begin{equation}\label{term3}
\begin{split}
\int_{\R^n} |f(x)|^2 w(x) dx&\leq 
320 \|w\|_{A_2(\text{Poisson},  \R^n)} \int_{\R^n}g^2_{*}(f)(x) w(x) dx\\
&\qquad +2\int_{\R^n}|K_af(z)|^2\E_{(z, a)} \left(\mathbf{1}_{B}(X_{\tau})w(X_{\tau}) \right)dz.
\end{split}
\end{equation}
Since $f\in C_{0}(\R^n)$, we have 
$$
|K_af(z)|=|\int_{\R^n}k_a(x-z)f(x)\,dx|\leq \frac{c_n}{a^{n}}\int_{\R^n}|f(x)| dx. 
$$
Thus, 
\begin{eqnarray*}
&&2\int_{\R^n}|K_af(z)|^2\E_{(z, a)} \left(\mathbf{1}_{B}(X_{\tau})w(X_{\tau}) \right)dz\\
&\leq& \frac{2c_n}{a^{n}}\left(\int_{\R^n}|f(x)| dx\right)\left(\int_{\R^n}\E_{(z, a)} \left(\mathbf{1}_{B}(X_{\tau})w(X_{\tau}) \right)dz\right)\\
&=& \frac{2c_n}{a^{n}}\left(\int_{\R^n}|f(x)| dx\right)\left(\int_{B}w(x)dx\right). 
\end{eqnarray*}
Combining this with \eqref{term3} and letting $a\to\infty$ gives 
$$
\int_{\R^n} |f(x)|^2 w(x) dx\leq 
320 \|w\|_{A_2(\text{Poisson},  \R^n)} \int_{\R^n}g^2_{*}(f)(x) w(x) dx, 
$$
which is the announced inequality.  


Similarly,  inequality \eqref{mainin4}  in Theorem \ref{contmart} gives 
\begin{align*}
&\int_{\R^n} g^2_{*, a}(f)(x) w(x) dx\\
&=\frac{1}{2}\int_{\R^n} \E_{(z, a)}\left(\int_{0}^{\tau} |\nabla K_{Y_s}(f)(X_s)|^2\,ds\, w(X_{\tau})\right)dz\\
&\leq  \frac{1}{2}\int_{\R^n} E_{(z, a)}\left(\left(|K_af(z)|^2+\int_{0}^{\tau} |\nabla K_{Y_s}(f)(X_s)|^2\,ds\right)\, w(X_{\tau})\right)dz\\
&\leq \frac{1}{2} \inf_{1<r<2}\left(\frac{r}{2-r}\|w\|_{A_2(\text{Poisson},  \R^n)}\right)\int_{\R^n}\E_{(z, a)}|f(X_{\tau}|^2 w(X_{\tau}) dz\\
&=\frac{1}{2}\inf_{1<r<2}\left(\frac{r}{2-r}\|w\|_{A_2(\text{Poisson},  \R^n)}\right)\int_{\R^n} |f(x)|^2 w(x)dx.
\end{align*}

Combining the above arguments with Corollary \ref{maincor1} we obtain \eqref{maininRn3} and this completes the proof of the theorem. 
\end{proof}

From the pointwise inequalities \eqref{pointwiseAbyg*} and \eqref{pointwisegbyg*}, combined with \eqref{maininRn3}, we obtain

\begin{corollary}  Suppose $w\in A_2(\text{Poisson},  \R^n)$ and $f\in C_{0}(\R^n)$.  Then 
\begin{equation}\label{maininRn4}
\|A_{\alpha}(f)(x)\|_{L_w^2(\R^n)}\leq \frac{(\alpha^2+1)^{\frac{n+1}{4}}}{\sqrt{c_n}} 2^{5/4}\|w\|_{A_2(\text{Poisson},  \R^n)}\|f\|_{L_w^2(\R^n)} 
\end{equation}
and 
\begin{equation}\label{maininRn5}
\|g(f)(x)\|_{L_w^2(\R^n)}\leq  2^{9/4}\|w\|_{A_2(\text{Poisson},  \R^n)}\|f\|_{L_w^2(\R^n)}.  
\end{equation}
\end{corollary}
 
Our results for $A_p(\text{heat})$ weights parallel those for $A_p(\text{Poisson},  \R^n)$.  We start with the corresponding lemma which shows the identification of these weights with the martingale weights arising from the semigroup.  

\begin{lemma}\label{equiheat1} Suppose $w\in A_p(\text{heat},  \R^n)$.  Fix $0<T<\infty$ and $z\in \R^n$. Let $(B_t)_{t\geq 0}$ be Brownian motion  in $\R^n$ starting at $z$.  Consider the martingale ${Y}_t=P_{T-t}w(B_t)$, $0\leq t\leq T$, under the measure $\p_z$.   Then ${Y}_{T}=w(B_T)\in A_p(\text{mart})$ and $\|Y_T\|_{A_{p}(\text{mart})}\leq \|w\|_{A_p(\text{heat},  \R^n)}$.
\end{lemma}
\begin{proof}  Recall that $P_tw(z)=\E_zw(B_t)$.  As before, we apply the Strong Markov property to obtain that for any $0<t<T$, 
\begin{eqnarray*}
\E_{z}\left(\left(\frac{1}{{Y}_{T}}\right)^{1/(p-1)}\big|\F_{t}\right)&=& \E_{z}\left(\left(\frac{1}{w(B_T)}\right)^{1/(p-1)}\big|\F_{t}\right)\\
&=& \E_{B_{t}}\left(\left(\frac{1}{w(X_{T-t})}\right)^{1/(p-1)}\right)\\
&=&P_{T-t}({w^{-1/(p-1}})(B_{t}). 
\end{eqnarray*} 
Thus,
$$
Y_t\left(\E_{z}\left(\left(\frac{1}{{Y}_{T}}\right)^{1/(p-1)}\big|\F_{t}\right)\right)^{p-1}=P_{T-t}w(B_t)\left(P_{T-t}({w^{-1/(p-1}})(B_{t})\right)^{p-1}
$$
and 
$$
\|Y_T\|_{A_{p}(\text{mart})}=\sup_{0<t<T}\|Y_t\left(\E_{z}\left(\left(\frac{1}{{Y}_{T}}\right)^{1/(p-1)}\big|\F_{t}\right)\right)^{p-1}\|_{\infty}
\leq \|w\|_{A_p(\text{heat},  \R^n)}, 
$$
as claimed. 
\end{proof}

As before, a remark similar to Remark \ref{remark1} applies.  With this lemma established, we can repeat the above argument for the space-time martingales and obtain similar results for 
$\|w\|_{A_p(\text{heat},  \R^n)}$ weights. 

\begin{theorem}\label{A_2(heat)}  Suppose $w\in A_2(\text{heat},  \R^n)$ and $f\in C_{0}(\R^n)$.  Then
\begin{equation}\label{mainnheatRn1}
||f||_{L_w^2(\R^n)}\leq \large\left(160\|w\|_{A_2(\text{heat},  \R^n)}\large\right)^{1/2}||G_*(f)||_{L_w^2(\R^n)},
\end{equation}
\begin{equation}\label{maininheatRn2}
||G_{*}(f)||_{L_w^2(\R^n)}\leq  \inf_{1<r<2}\left(\frac{r}{2-r}\|w\|_{A_2(\text{heat},  \R^n)}\right)^{1/2}\|f\|_{L_w^2(\R^n)}
\end{equation}
and 
\begin{equation}\label{maininheatRn3}
\|G_*(f)\|_{L_w^2(\R^n)}\leq 2^{7/4}\|w\|_{A_2(\text{heat},  \R^n)}\|f\|_{L_w^2(\R^n)}.
\end{equation}
\end{theorem}

We remark that the reason the constants here are slightly different than those for the Poisson case is that the representation for $G_{*, T}$ in terms of the conditional expectation of the corresponding martingale square function given in \eqref{Cond-Expect1} does not have the $\frac{1}{2}$ factor as in \eqref{g*harm1}.  

From the inequalities 
$\calP A_\alpha (f)(x)\leq (2\pi)^{n/4}e^{\alpha^2/4} G_{*}(f)(x)
$ and
$G(f)(x)\leq \sqrt{2}G_{*}(f)(x)$ 
 already proved in \eqref{pointwiseAbyG*} and \eqref{pointwiseGbyG*}, we obtain 

\begin{corollary}\label{heat-Cor1}  Suppose $w\in A_2(\text{heat},  \R^n)$ and $f\in C_{0}(\R^n)$.  Then 
\begin{equation}\label{maininheatRn4}
\|\calP A_\alpha (f)(x)\|_{L_w^2(\R^n)}\leq (2\pi)^{n/4}e^{\alpha^2/4} 2^{7/4}\|w\|_{A_2(\text{heat},  \R^n)}\|f\|_{L_w^2(\R^n)}
\end{equation}
and 
\begin{equation}\label{maininheatRn5}
\|G(f)(x)\|_{L_w^2(\R^n)}\leq  2^{9/4}\|w\|_{A_2(\text{heat},  \R^n)}\|f\|_{L_w^2(\R^n)}.  
\end{equation}
\end{corollary}

Theorem \ref{A_2(heat)} and Corollary \ref{heat-Cor1}, combined with the Petermichl--Volberg inequality \eqref{PetVol} proving the equivalence of the classical Muckenhoupt $A_p$ and $A_p(\text{heat})$, give

\begin{corollary}\label{heat-Cor2}  Suppose $w\in A_2$ is in the classical Muckenhoupt  class and $f\in C_{0}(\R^n)$.  Then
\begin{equation}\label{classical1}
||f||_{L_w^2(\R^n)}\leq \large\left(\frac{160\|w\|_{A_2}}{a}\large\right)^{1/2}||G_*(f)||_{L_w^2(\R^n)},
\end{equation}
\begin{equation}\label{classical2}
||G_{*}(f)||_{L_w^2(\R^n)}\leq  \inf_{1<r<2}\left(\frac{r}{2-r}\frac{\|w\|_{A_2}}{a}\right)^{1/2}\|f\|_{L_w^2(\R^n)}, 
\end{equation}
\begin{equation}\label{classical3}
\|G_*(f)\|_{L_w^2(\R^n)}\leq \frac{2^{7/4}\|w\|_{A_2}}{a}\|f\|_{L_w^2(\R^n)}, 
\end{equation}
\begin{equation}\label{classical4}
\|\calP A_\alpha (f)(x)\|_{L_w^2(\R^n)}\leq \frac{(2\pi)^{n/4}e^{\alpha^2/4} 2^{7/4}\|w\|_{A_2}}{a}\|f\|_{L_w^2(\R^n)} 
\end{equation}
and 
\begin{equation}\label{classical5}
\|G(f)(x)\|_{L_w^2(\R^n)}\leq  \frac{2^{9/4}\|w\|_{A_2}}{a}\|f\|_{L_w^2(\R^n)},   
\end{equation}
where $a$ is the constant in \eqref{PetVol}. 
\end{corollary}

\section{symmetric Markovian semigroups}\label{Markovian} 
In \cite{Var1}, Varopoulos  defines the $g_{*}$ function in the general setting of Poisson semigroups.    However, due to the lack of gradient in this general setting, he only considers the time derivative of the  semigroup in the definition of his square functions, for both his $g$ and $g_{*}$.  This construction can be applied to obtain versions of the above inequalities for semigroups which yield martingales with continuous paths. In this section we aim to define a  Littlewood-Paley function $G_{*}$ for general Markovian and the corresponding $A_p$ weights.  Since our martingale results require continuous  trajectories, our $A_2$ inequality will only be stated for Riemannian manifolds of non-negative Ricci curvature,   using Meyer's ``carre du champ." 

Let $\trip$ be a metric measure space.  That is,  a  measure space (equipped with a countably generated $\sigma$-algebra) $\CM$ which is also a metric space with metric $\rho$. The measure $\mu$ is assumed to be $\sigma$-finite.   
Let $(P_{t}, t \geq 0)$ be a family of Markovian linear operators which acts as a $C_{0}$-contraction semigroup on $L^{p}(M)$ for all $1 \leq p \leq \infty$. 
We further assume that $P_{t}$ is self-adjoint on $L^{2}(M)$ for all $t \geq 0$ and that it is given by an integral kernel 
$$
P_tf(x)=\int_{M}p_t(x, y) f(y) d\mu(y)
$$ 
which is symmetric.  That is, $p_t(x, y)=p_t(y, x)$ and 
 $$
 \int_{M}p_t(x, y)d\mu(y)=1. 
 $$
 It follows from \cite{Davie} that $T_t=e^{-t\CL}$ where $\CL$ is a positive self-adjoint operator on $L^2(M)$. If we denote by $D(\CL)\subset L^2(M)$ the domain of $\CL$ 
the for $f, h\in D(\CL)$, define the operator ``carr\'e du champ" is defined by
\begin{equation}\label{carre}
\Gamma(f, h)=\CL(fh)-f\CL h-h\CL f.
\end{equation}
By the definition  of $\Gamma$, we have 
\begin{eqnarray*}
\int_{0}^{\infty}\int_{M}\Gamma(P_tf, P_tf)(x)d\mu(x)dt&=&-2\int_{0}^{\infty}\int_{M}P_tf(x)\CL P_tf(x)d\mu(x) dt\\
&=&-2\int_{0}^{\infty}\int_{M}P_tf(x)\frac{d}{dt}P_tf(x)  d\mu(x)dt\\
&=&-\int_{0}^{\infty}\int_{M}\frac{d}{dt}(P_tf(x))^2 d\mu(x) dt\\
&=& \int_{M}|f(x)|^2d\mu(x).
\end{eqnarray*}
Defining, respectively,  the Littlwood-Paley $G$ and $G_{*}$ by 
\begin{equation}\label{G-general1}
G(f)(x)=\left(\int_0^{\infty} \Gamma(P_tf, P_tf)(x)\,dt\right)^{1/2} 
\end{equation}
and 
\begin{equation}\label{G*-genral1}
G_{*}(f)(x)=\left(\int_0^{\infty}\int_{M}\Gamma(P_tf, P_tf)(x)p_t(x, z)d\mu(z)\, dt\right)^{1/2},
\end{equation}
 we see that 
$$\|G_{*}f\|_2=\|G(f)\|_2=\|f\|_2.
$$
We now denote by $(X_t)_{t\geq 0}$ the Markov process associated with this semigroup so that $P_t f(x)=\E_{x}[f(X_t)]$, and consider the martingale  
 \begin{equation}\label{manifoldmart}
 M_t(f)=P_{T-t}f(X_t),\,\,\,\, 0\leq t\leq T. 
 \end{equation} 
Under quite general conditions on the Markovian semigroup (as those imposed on \cite{Var1}), the process $(X_t)_{t\geq 0}$ has {\cadlag} paths, $G_{*}$ is the conditional expectation of the square function for this martingale (see
Bakry Emery \cite[p.~181]{BakEme1} or Revuz and Yor, \cite[p. 326]{RevYor1}) and it follows from the Burkholder-Gundy  inequalities that $\|G_{*}f(x)\|_p\leq C_p\|f\|_p$, $2\leq p<\infty$, where $C_p$ depends only on $p$.  

For the remaining of this paper we will make the further assumption that our Markovian semigroup corresponds to Brownian motion on a complete Riemannian manifold of non-negative Ricci curvature and therefore the process has continuous paths.  To be precise, we let $M$ be a complete Riemannian manifold  of dimension $n$ with non-negative Ricci curvature. Let $\Delta$ be the Laplace-Beltrami operator and $\mu$ be the Riemannian volume measure. Then the heat equation $\frac{\partial u}{\partial t} = \Delta u(t)$  has a fundamental solution $p \in C^{\infty}((0, \infty) \times M \times M)$ which we call the {\it heat kernel} and this gives the kernel generating our semigroup $(P_t)_{t\geq 0}$ above. 
The following heat kernel bounds of Li and Yau \cite{LY} are important for many applications.
For all $t > 0, x,z \in M$:
\begin{equation} \label{hk}
\frac{C_1}{V(x,\sqrt{t})}\exp\left(-\frac{\rho(x,y)^{2}}{c_1t}\right)\leq p_{t}(x,y)\leq \frac{C_2}{V(x,\sqrt{t})}\exp\left(-\frac{\rho(x,y)^{2}}{c_3t}\right),
\end{equation}
where $\rho$ is the Riemannian metric and for $r > 0, V(x,r)=\mu(B(x, r))$ is the volume of the ball $B(x, r)$ of radius  $r$ centered at $x$.
It is also well-known (cf. \cite{BC}) that for all  $x \in M$,
\begin{equation}\label{volumeright}
V(x,r) \leq v(n)r^{n},
\end{equation} 
where $v(n)$ is the volume of the unit ball in $\R^{d}$.  

With the Laplacian as the generator, the carr\'e du champ has the familiar form 
$$\Gamma(P_tf, P_tf)(x)=|\nabla P_tf(x)|^2$$
and the square function of the martingale $M_f(f)$ is given by 
$$
\langle M(f)\rangle_t=|P_Tf(x)|^2+ \int_0^t|\nabla_x P_{T-s}f(X_s)|^2\,ds,\qquad t\geq 0.  
$$ 
With this, the exact same argument as in $\R^n$ gives that 
\begin{eqnarray}\label{manifoldG*}
G_{*, T}(f)(x)&=&\int_0^T \int_{\R^d} |\nabla_x P_{t}f(z)|^2 p_t(x, z)d\mu(z)\,dt\nonumber\\
&=& \int_{\R^n} \E_{z}^{x}\left(\int_{0}^{T}|\nabla_x P_{T-s}f(X_s)|^2\,ds\right)p_{T}(x, z)d\mu(z).
\end{eqnarray}
As before we have the pointwise inequality 
\begin{equation}\label{domGbyG*-manifold} 
G(f)(x)\leq \sqrt{2} G_{*}(f)(x).
\end{equation} 
To prove this we recall that under the assumption of non-negative Ricci curvature, the  ``Bakry $\Gamma_2\geq 0$" holds.  That is, we have the inequality $\Gamma(P_tf, P_tf)\leq P_t\Gamma(f, f)$ (see \cite{Bak1} for details).  From this and the semigroup property, we obtain 
\begin{eqnarray*}
\Gamma(P_tf, P_tf)(x)&=&\Gamma(P_{t/2}T_{t/2}f, P_{t/2}P_{t/2}f)(x)\\
&\leq& P_{t/2}\Gamma(T_{t/2}f, P_{t/2}f)(x)\\
&=&\int_{M}\Gamma(T_{t/2}f, P_{t/2}f)(y) p_{t/2}(x, y) d\mu(y). 
\end{eqnarray*}
Integrating both sides of this inequality in $t$ gives \eqref{domGbyG*-manifold}. 

Next, we introduce the parabolic cone using the metric on the manifold by 
 $$
\calP_\alpha(x)=\{(z,t)\in \R_+^{n+1}:d(x, z) < \alpha \sqrt{t}\} 
$$ 
and define the Lusin area integral by 
\begin{equation}\label{Lus-heat-M}
\calP A_\alpha (f)(x)=\left( \int_{\calP_\alpha(x)} t^{-{n/2}}|\nabla P_{t}f(z)|^2 d\mu(z)
dt\right)^{1/2}.
\end{equation} 
By \eqref{hk} and \eqref{volumeright} we have 
\begin{equation}\label{dominAbyG*-M}
\calP A_\alpha (f)(x)\leq \sqrt{\frac{v(n)e^{\frac{\alpha^2}{2c_1}}}{C_1}}G_{*}(f)(x).
\end{equation} 

Given a positive and $\mu$-locally integrable function $w$ on $M$, we will write $w\in A_p(heat, M)$ 
 if
\begin{equation}\label{A_p-heat-M}
\|w\|_{A_p(\text{heat}, M)}:=\|P_tw(x) \left(P_t({w^{-1/(p-1}})(x)\right)^{p-1}\|_{L^{\infty}(M\times (0, \infty))}<\infty. 
\end{equation}
The same argument as that in Lemma \ref{equiheat1} shows that $$\|Y_T\|_{A_{p}(\text{mart})}\leq \|w\|_{A_p(\text{heat}, M)},$$
where $Y$ stands for the martingale  $Y_t=P_{T-t}w(X_t)$, $0\leq t\leq T$.

 Similarly, we say that $w\in A_p(M)$ (the classical Muckenhoupt $A_p$-class)  if 
 
$$ \|w\|_{A_p}(M)=\sup_B \left(\frac{1}{\mu(B)}\int_B w(z)d\mu(z) \right)\left(\frac{1}{\mu(B)}\int_B w(z)^{-1/(p-1)} d\mu(z)\right)^{p-1}<\infty,$$
where the supremum is taking over balls.  
 Because of the bound on the heat kernel $p_t(x, z)$ given in \eqref{hk}, the observations of Remark \ref{A_pbyApoisson/heat} show   that 
 $$A_p(M)\leq aA_p(heat, M),$$ for some constant $a$ depending on $c_1, C_1$.

 With the above definitions in place, we now state the following version of Theorem \ref{A_2(heat)}, whose proof is exactly the same as the proof of that theorem.   
 
 \begin{theorem}\label{heat-M} Let $M$ be a complete Riemannian manifold of non-negative Ricci curvature.  Assume further that (*) $\sup_{x\in M}p_t(x, x)=c_t\to 0$, as $t\to\infty$, holds true.  Suppose $w\in A_2(heat, M)$ and $f\in C_0(M)$.  Then inequalities \eqref{mainnheatRn1}, \eqref{maininheatRn2}, \eqref{maininheatRn3}, \eqref{maininheatRn4} and \eqref{maininheatRn5} hold for the functions $G(f)$, $G_{*}(f)$ and $\calP A_\alpha (f)$ as defined in \eqref{G-general1}, \eqref{G*-genral1} and \eqref{Lus-heat-M}.  
 \end{theorem}
 
 \begin{remark}  If in addition we assume that $V(x, r)\geq c_nr^n$, then  (*) is automatically satisfied.  For various known conditions that guarantee this lower bound volume growth, see \cite[p.~255]{Var2}) and \cite{AusCouDuoHof}. 
  \end{remark}

\section*{Acknowledgment} The research was initiated during the fall semester of 2013 when the second-named author visited Purdue University.

\end{document}